\documentclass[12pt]{article}
\usepackage{amsthm}
\usepackage{amsfonts}
\usepackage{amsmath}
\usepackage{amscd}
\usepackage{amssymb}
\usepackage{pstricks,pstricks-add,pst-math,pst-xkey}
\usepackage{stmaryrd}
\usepackage{newalg}
\theoremstyle{plain}
\newtheorem{thm}{Theorem}[section]
\newtheorem{lem}[thm]{Lemma}
\newtheorem{prop}[thm]{Proposition}
\newtheorem{cor}[thm]{Corollary}
\newtheorem{prob2}{Problem}

\theoremstyle{definition}

\newtheorem*{prob}{Problem}

\theoremstyle{remark}

\newcommand{\N}{\mathbb{N}}
\newcommand{\Z}{\mathbb{Z}}
\newcommand{\Zn}{\Z/n\Z}
\newcommand{\m}{\mathfrak{m}}

\newcommand{\ord}{\mathcal{O}}

\newcommand{\twoheadlongrightarrow}{\relbar\joinrel\twoheadrightarrow}
\DeclareMathOperator*{\lcm}{lcm}
\DeclareMathOperator*{\rad}{rad}
\renewcommand{\leq}{\leqslant}
\renewcommand{\geq}{\geqslant}
\title{ON A PROBLEM OF MOLLUZZO CONCERNING STEINHAUS TRIANGLES IN FINITE CYCLIC GROUPS}
\date{07/09/08}
\author{Jonathan \textsc{Chappelon}}
\begin{document}
\maketitle

\begin{abstract}
Let $X$ be a finite sequence of length $m\geq 1$ in $\Zn$. The \textit{derived sequence} $\partial X$of $X$ is the sequence of length $m-1$ obtained by pairwise adding consecutive terms of $X$. The collection of iterated derived sequences of $X$, until length $1$ is reached, determines a triangle, the \textit{Steinhaus triangle $\Delta X$ generated by the sequence $X$}. We say that $X$ is \textit{balanced} if its Steinhaus triangle $\Delta X$ contains each element of $\Zn$ with the same multiplicity. An obvious necessary condition for $m$ to be the length of a balanced sequence in $\Zn$ is that $n$ divides the binomial coefficient $\binom{m+1}{2}$. It is an open problem to determine whether this condition on $m$ is also sufficient. This problem was posed by Hugo Steinhaus in 1963 for $n=2$ and generalized by John C. Molluzzo in 1976 for $n\geq3$. So far, only the case $n=2$ has been solved, by Heiko Harborth in 1972. In this paper, we answer positively Molluzzo's problem in the case $n=3^k$ for all $k\geq1$. Moreover, for every odd integer $n\geq3$, we construct infinitely many balanced sequences in $\Zn$. This is achieved by analysing the Steinhaus triangles generated by arithmetic progressions. In contrast, for any $n$ even with $n\geq4$, it is not known whether there exist infinitely many balanced sequences in $\Zn$. As for arithmetic progressions, still for $n$ even, we show that they are never balanced, except for exactly 8 cases occurring at $n=2$ and $n=6$.
\end{abstract}

\section{Introduction}

Let $\Zn$ denote the finite cyclic group of order $n\geq1$. Let $X=(x_1,x_2,\ldots,x_m)$ be a sequence of length $m\geq2$ in $\Zn$. We define the \textit{derived sequence} $\partial X$ of $X$ as
$$
\partial X=(x_1+x_2,x_2+x_3,\ldots,x_{m-1}+x_m),
$$
where $+$ is the sum in $\Zn$. This is a finite sequence of length $m-1$. Iterating the derivation process, we denote by $\partial^iX$ the $i$th derived sequence of $X$, defined recursively as usual by $\partial^0X=X$ and $\partial^iX=\partial(\partial^{i-1}X)$ for all $i\geq1$. Then, the $i$th derived sequence of $X$ can be expressed by means of the elements of $X=(x_1,x_2,\ldots,x_m)$ as follows
$$
\partial^iX=\left(\sum_{k=0}^{i}{\binom{i}{k}x_{1+k}},\sum_{k=0}^{i}{\binom{i}{k}x_{2+k}},\ldots,\sum_{k=0}^{i}{\binom{i}{k}x_{m-i+k}}\right),
$$
for every $0\leq i\leq m-1$, where $\binom{i}{k}=\frac{i!}{k!(i-k)!}$ denotes the binomial coefficient for $0\leq k\leq i$.
\par The \textit{Steinhaus triangle $\Delta X$ generated by the sequence $X=(x_1,x_2,\ldots,x_m)$} is the multiset union, where all occurring multiplicities are added, of all iterated derived sequences of $X$, that is,
$$
\Delta X = \bigcup_{i=0}^{m-1}{\partial^iX} = \left\{ \sum_{k=0}^{i}{\binom{i}{k}x_{j+k}}\ \middle| \ 0\leq i\leq m-1,\ 1\leq j\leq m-i\right\}.
$$
Note that the Steinhaus triangle generated by a sequence of length $m\geq1$ is composed by $\binom{m+1}{2}$ elements of $\Zn$, counted with their multiplicity. For example, the Steinhaus triangle $\Delta X$ generated by the sequence $X=(0,1,2,2)$ in $\Z/3\Z$ can be represented as Figure~\ref{fig1}, where the $i$th row of the triangle is the $(i-1)$th derived sequence $\partial^{i-1}X$ of $X$.

\begin{figure}[!h]
\centering
\begin{pspicture}(3,1.9485)
\pspolygon(0.375,1.9485)(2.625,1.9485)(1.5,0)
\rput(1.5,0.433){$2$}
\rput(1.25,0.866){$1$}
\rput(1.75,0.866){$1$}
\rput(1,1.299){$1$}
\rput(1.5,1.299){$0$}
\rput(2,1.299){$1$}
\rput(0.75,1.732){$0$}
\rput(1.25,1.732){$1$}
\rput(1.75,1.732){$2$}
\rput(2.25,1.732){$2$}
\end{pspicture}
\caption{\label{fig1}A Steinhaus triangle in $\Z/3\Z$}
\end{figure}
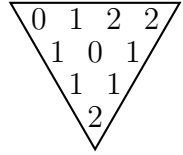

A finite sequence $X$ in $\Zn$ is said to be \textit{balanced} if each element of $\Zn$ occurs in the Steinhaus triangle $\Delta X$ with the same multiplicity. For instance, the sequence $(2,2,3,3)$ is balanced in $\Z/5\Z$. Indeed, as depicted in Figure~\ref{fig2}, its Steinhaus triangle is composed by each element of $\Z/5\Z$ occurring twice. Note that, for a sequence $X$ of length $m\geq1$ in $\Zn$, a necessary condition to be balanced is that the integer $n$ divides the binomial coefficient $\binom{m+1}{2}$, the cardinality of the Steinhaus triangle $\Delta X$.

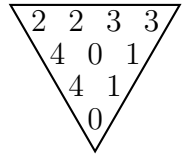
\begin{figure}[!h]
\centering
\begin{pspicture}(3,1.9485)
\pspolygon(0.375,1.9485)(2.625,1.9485)(1.5,0)
\rput(1.5,0.433){$0$}
\rput(1.25,0.866){$4$}
\rput(1.75,0.866){$1$}
\rput(1,1.299){$4$}
\rput(1.5,1.299){$0$}
\rput(2,1.299){$1$}
\rput(0.75,1.732){$2$}
\rput(1.25,1.732){$2$}
\rput(1.75,1.732){$3$}
\rput(2.25,1.732){$3$}
\end{pspicture}
\caption{\label{fig2}The Steinhaus triangle of a balanced sequence in $\Z/5\Z$}
\end{figure}

This concept was introduced by Hugo Steinhaus in 1963 \cite{steinbinprob}, who asked whether there exists, for each integer $m\equiv 0$ or $3 \pmod4$ (i.e. whenever the binomial coefficient $\binom{m+1}{2}$ is even), a balanced binary sequence of length $m$, i.e. a sequence of length $m$ in $\Z/2\Z$ whose Steinhaus triangle contains as many $0$'s as $1$'s. This problem was answered positively for the first time by Heiko Harborth in 1972 \cite{harbansw} by showing that, for every $m\equiv 0$ or $3\pmod4$, there exist at least four balanced binary sequences of length $m$. New solutions of the Steinhaus's problem recently appeared in \cite{eliahach2}, \cite{eliahach1} and \cite{elia3}. The possible number of ones in a binary Steinhaus triangle was explored in \cite{chang}. In 1976, John C. Molluzzo \cite{molluzzo} extended the definition of Steinhaus triangle to any finite cyclic group $\Zn$ and he posed the generalization of the Steinhaus's original problem.

\begin{prob}[Molluzzo, 1976]
Let $n$ be a positive integer. Given a positive integer $m$, is it true that there there exists a balanced sequence of length $m$ in $\Zn$ if and only if the binomial coefficient $\binom{m+1}{2}$ is divisible by $n$?
\end{prob}

This generalization of the Steinhaus's original problem corresponds to Question 8 of \cite{dymacek}. So far this problem was completely open for $n\geq3$. In this paper, we solve in the affirmative the case $n=3^k$ for all $k\geq1$. Moreover, we show that there exist infinitely many balanced sequences in each finite cyclic group of odd order.

This paper is organized as follows. In Section 2, we present generalities on balanced sequences in finite cyclic groups. In Section 3, we describe the structure of the Steinhaus triangle generated by an arithmetic progression in $\Zn$. This permits us to show, in Section 4, that there exists a positive integer $\alpha(n)$, for each odd number $n$, such that every arithmetic progression with invertible common difference and of length $m\equiv 0$ or $-1\pmod{\alpha(n)n}$ is a balanced sequence in $\Zn$. This result is refined in Section 5, by considering antisymmetric sequences. Particularly, this refinement answers in the affirmative Molluzzo's Problem in $\Z/3^k\Z$ for all $k\geq1$. In contrast with the results obtained in Sections 4 and 5, we show, in Section 6, that the arithmetic progressions in finite cyclic groups of even order $n$ are never balanced, except for exactly 8 cases occurring at $n=2$ and $n=6$. Finally, in Section 7, we conclude with several remarks and open subproblems of Molluzzo's problem.

\section{Generalities on balanced sequences}

In this section, we will establish the admissible lengths of balanced sequences in $\Zn$ and study the behaviour of balanced sequences under projection maps.

\subsection{On the length of a balanced sequence}
The set of all prime numbers is denoted by $\mathcal{P}$. For every prime number $p$, we denote by $v_p(n)$ the $p$-adic valuation of $n$, i.e. the greatest exponent $e\geq0$ for which $p^e$ divides $n$. The prime factorization of $n$ may then be written as
$$
n=\prod_{p\in\mathcal{P}}p^{v_p(n)}.
$$
We denote by $\omega(n)$ the number of distinct prime factors of $n$, i.e. the number of primes $p$ for which $v_p(n)\geq1$.

As seen in Section 1, a necessary condition for a sequence of length $m\geq1$ in $\Zn$ to be balanced is that $n$ divides the binomial coefficient $\binom{m+1}{2}$, the cardinality of the Steinhaus triangle $\Delta X$. The set of all positive integers $m$ satisfying this divisibility condition is described in the following theorem.

\begin{thm}\label{thm6}
Let $n$ be a positive integer. The set of all positive integers $m$ such that the binomial coefficient $\binom{m+1}{2}$ is a multiple of $n$ is a disjoint union of $2^{\omega(n)}$ distinct classes modulo $2n$ if $n$ is even, and of the same number of distinct classes modulo $n$ if $n$ is odd. This set comprises the classes $2n\N$ and $(2n-1)+2n\N$ if $n$ is even, and the classes $n\N$ and $(n-1)+n\N$ if $n$ is odd.
\end{thm}

\begin{proof}
Let $n$ and $m$ be two positive integers. Then,
$$
\binom{m+1}{2}\equiv0\pmod{n}
\begin{array}[t]{l}
\Longleftrightarrow m(m+1)\equiv0\pmod{2n}\\
\Longleftrightarrow \left\{
\begin{array}{l}
m(m+1)\equiv0\pmod{2^{v_2(n)+1}}\\
m(m+1)\equiv0\pmod{p^{v_p(n)}},\ \forall p\in\mathcal{P}\setminus\{2\}
\end{array}\right.\\
\Longleftrightarrow \left\{
\begin{array}{l}
m\equiv a_2\pmod{2^{v_2(n)+1}}\\
m\equiv a_p\pmod{p^{v_p(n)}},\ \forall p\in\mathcal{P}\setminus\{2\}
\end{array}\right.
\end{array}
$$
with $a_p\in\{-1,0\}$ for every prime $p$. Each integer $m$ of the set appears then as a solution of a system of congruences composed by $\omega(n)$ non-trivial equations. By the Chinese remainder theorem there exists a unique solution modulo $n$ or modulo $2n$ according to the parity of $n$. This permits us to conclude that there exist $2^{\omega(n)}$ such classes modulo $2n$ (resp. modulo $n$) for every even (resp. odd) number $n$. Particularly, if $n$ is even (resp. odd) and $a_p=0$ for every prime $p$, then the positive integers $m$, such that the binomial $\binom{m+1}{2}$ is a multiple of $n$, constitute the class $2n\N$ (resp. the class $n\N$). By the same way, if $n$ is even (resp. odd) and $a_p=-1$ for every prime $p$, then such positive integers $m$ constitute the class $(2n-1)+2n\N$ (resp. the class $(n-1)+n\N$).
\end{proof}

\begin{cor}\label{cor1}
Let $p$ be an odd prime number and $k$ be a positive integer. For every positive integer $m$, we have
$$
\binom{m+1}{2}\equiv0\pmod{p^k} \Longleftrightarrow m\equiv0\ \text{or}\ -1\pmod{p^k}.
$$
Similarly, for every positive integer $m$, we have
$$
\binom{m+1}{2}\equiv0\pmod{2^k} \Longleftrightarrow m\equiv0\ \text{or}\ -1\pmod{2^{k+1}}.
$$
\end{cor}

For instance, for $n=825=3\cdot5^2\cdot11$, the set of positive integers $m$ such that the binomial coefficient $\binom{m+1}{2}$ is divisible by $825$ is the disjoint union of the $8$ classes $a+825\N$ with $a\in\{0,99,275,374,450,549,725,824\}$.

\subsection{Balanced sequences under projection maps}

For every finite multiset $M$ of $\Zn$, we define and denote by $\m_M$ the multiplicity function of $M$ as the function
$$
\m_M : \Zn \longrightarrow \N
$$
which assigns to each element $x$ in $\Zn$ the number of occurrence $\m_M(x)$ of $x$ in the multiset $M$. We agree that the multiplicity function $\m_M$ vanishes at every $x$ not in $M$.
\par As usual, the cardinality $|M|$ of a finite multiset $M$ is the total number of elements in $M$, counted with multiplicity, that is,
$$
|M|=\sum_{x\in\Zn}\m_M(x).
$$
Let $X$ be a sequence of length $m\geq1$ in $\Zn$. Since the Steinhaus triangle $\Delta X$ is a multiset of cardinality $\binom{m+1}{2}$, it follows that the sequence $X$ is balanced if, and only if, the multiset $\Delta X$ has a constant multiplicity function $\m_{\Delta X}$ equal to $\frac{1}{n}\binom{m+1}{2}$.

For every factor $q$ of the positive integer $n$, we denote by $\pi_q$ the canonical surjective morphism $\pi_q : \Zn\twoheadlongrightarrow\Z/q\Z$. For a finite sequence $X=(x_1,x_2,\ldots,x_m)$ of length $m\geq1$ in $\Zn$, we define, and denote by
$$
\pi_q(X)=\left( \pi_q(x_1),\pi_q(x_2),\ldots,\pi_q(x_m) \right),
$$
its projected sequence in $\Z/q\Z$. We now study the behaviour of balanced sequences in $\Zn$ under the projection morphism $\pi_q : \Zn\twoheadlongrightarrow\Z/q\Z$.

\begin{thm}\label{thm2}
Let $q$ be a divisor of $n$ and $X$ be a sequence of length $m\geq1$ in $\Zn$. Then, the sequence $X$ is balanced if, and only if, its projected sequence $\pi_q\left(X\right)$ is also balanced and the multiplicity function $\m_{\Delta X} : \Zn \longrightarrow \N$ is constant on each coset of the subgroup $q\Zn$.
\end{thm}

\begin{proof}
For every $x$ in $\Zn$, it is clear that the multiplicity of $\pi_q(x)$ in $\Delta\pi_q\left(X\right)$ is the sum of the multiplicities in $\Delta X$ of all the elements of the coset $x+q\Zn$, that is,
$$
\m_{\Delta \pi_q(X)}(\pi_q(x)) = \sum_{k=0}^{\frac{n}{q}-1}{\m_{\Delta X}(x+kq)},\ \forall x\in\Zn.
$$
This completes the proof.
\end{proof}

\section{Steinhaus triangles of arithmetic progressions}

In this section we will describe the structure of the Steinhaus triangle associated to an arithmetic progression of $\Zn$. We denote by
$$
AP(a,d,m)=\left( a,a+d,a+2d,\ldots,a+(m-1)d \right)
$$
the arithmetic progression beginning with $a\in\Zn$, with common difference $d\in\Zn$ and of length $m\geq1$. We begin by analysing the iterated derived sequences of an arithmetic progression in $\Zn$. First, its derived sequence is also an arithmetic progression in $\Zn$. More precisely, we have

\begin{prop}\label{prop4}
Let $n$ be a positive integer and let $a$ and $d$ be in $\Zn$. Then, the $i$th derived sequence of the arithmetic progression $AP(a,d,m)$ is the arithmetic progression
$$
\partial^i AP(a,d,m) = AP\left(2^ia+2^{i-1}id,2^id,m-i\right),
$$
for every $0\leq i\leq m-1$.
\end{prop}

\begin{proof}
If we set $X=AP(a,d,m)=(x_1,x_2,\ldots,x_m)$ and $\partial^i X=(y_1,y_2,\ldots,y_{m-i})$, then we have
$$
y_j
\begin{array}[t]{l}
= \displaystyle\sum_{k=0}^i\binom{i}{k}x_{j+k} = \sum_{k=0}^i\binom{i}{k}\left(a+(j+k-1)d\right) = \sum_{k=0}^i\binom{i}{k}\left(a+(j-1)d\right) + \sum_{k=0}^i\binom{i}{k}kd\\
= 2^i\left(a+(j-1)d\right) + 2^{i-1}id = \left( 2^ia+2^{i-1}id \right) + (j-1)2^id,
\end{array}
$$
for all $1\leq j\leq m-i$.
\end{proof}

For every sequence $X$ of length $m\geq1$ in $\Zn$, we denote by ${\Delta X}(i,j)$ the $j$th element of the $i$th row of the Steinhaus triangle $\Delta X$, i.e. the $j$th element of the $(i-1)$th derived sequence $\partial^{i-1}X$ of $X$, for all $1\leq i\leq m$ and all $1\leq j\leq m-i+1$. For example, in this notation, the $j$th element of the sequence $X$ is $\Delta X(1,j)$.

We now describe the coefficients of the Steinhaus triangle generated by an arithmetic progression in $\Zn$.

\begin{prop}\label{prop5}
Let $n$ be a positive integer. Let $a$ and $d$ be in $\Zn$ and $X=AP(a,d,m)$ be an arithmetic progression in $\Zn$. Then, we have
$$
\left\{
\begin{array}{ll}
{\Delta X}(1,j) = a+(j-1)d & ,\ \forall 1\leq j\leq m,\\
{\Delta X}(i,j) = 2^{i-1}a+2^{i-2}(2j+i-3)d & ,\ \forall 2\leq i\leq m,\ \forall 1\leq j\leq m-i+1.
\end{array}
\right.
$$
\end{prop}

\begin{proof}
This is merely a reformulation of Proposition \ref{prop4} using the notation $\Delta X(i,j)$ introduced above.
\end{proof}

Let $X$ be a finite sequence in $\Zn$. Every finite sequence $Y$ such that $\partial Y=X$ is called \textit{primitive sequence of $X$}. By definition of the derivation process, each finite sequence admits exactly $n$ primitives. However, for $n$ odd, if $X$ is an arithmetic progression in $\Zn$, then there is exactly one primitive of $X$ which is itself an arithmetic progression.

\begin{prop}\label{prop1}
Let $n$ be an odd number and let $a$ and $d$ be in $\Zn$. Then, the sequence $AP(2^{-1}a-2^{-2}d,2^{-1}d,m+1)$ is the only arithmetic progression whose derived sequence is the arithmetic progression $AP(a,d,m)$.
\end{prop}

\begin{proof}
By Proposition \ref{prop4}, the derived sequence of $AP(2^{-1}a-2^{-2}d,2^{-1}d,m+1)$ is the arithmetic progression $AP(a,d,m)$, that is,
$$
\partial AP(2^{-1}a-2^{-2}d,2^{-1}d,m+1)=AP(a,d,m).
$$
Suppose now that the arithmetic progressions $AP(a_1,d_1,m+1)$ and $AP(a_2,d_2,m+1)$ have the same derived sequence, that is,
$$
\partial AP(a_1,d_1,m+1) = \partial AP(a_2,d_2,m+1).
$$
Then, by Proposition \ref{prop4}, we have
$$
AP(2a_1+d_1,2d_1,m) = AP(2a_2+d_2,2d_2,m).
$$
It follows that $2a_1+d_1=2a_2+d_2$ and $2d_1=2d_2$. Since $2$ is invertible in $\Zn$, this leads to the equalities $a_1=a_2$ and $d_1=d_2$ and so the unicity of the statement is proved.
\end{proof}

In contrast, for $n$ even, there is no such unicity statement. For example, in $\Z/8\Z$, the arithmetic progressions $(3,7,3,7,3)$ and $(1,1,1,1,1)$ are distinct but have the same derived sequence $(2,2,2,2)$.

\section{Balanced arithmetic progressions in $\Zn$ for $n$ odd}

The integer $n$ is assumed to be odd throughout this section. We begin by showing that the common difference of a balanced arithmetic progression in $\Zn$ must be invertible.

\begin{thm}\label{thm3}
Let $n$ be an odd number and let $a$ and $d$ be in $\Zn$. If $d$ is non-invertible, then the arithmetic progression $AP(a,d,m)$ is not balanced for every positive integer $m$.
\end{thm}

\begin{proof}
Ab absurdo, suppose that there exists a balanced arithmetic progression
$$
X=AP(a,d,m)
$$
with non-invertible common difference $d$ in $\Zn$. We set
$$
q=\gcd(n,d_0)\neq1
$$
where $d_0$ is any integer whose residue class modulo $n$ is $d$. We consider the canonical surjective morphism $\pi_q : \Zn \twoheadlongrightarrow\Z/q\Z$ and the arithmetic progression
$$
\pi_q(X)=AP(\pi_q(a),\pi_q(d),m)=AP(\pi_q(a),0,m),
$$
which is a constant sequence in $\Z/q\Z$. Theorem~\ref{thm2} implies that the sequence $\pi_q(X)$ is balanced. Therefore there exists at least one coefficient in the Steinhaus triangle $\Delta \pi_q(X)$ which is zero, say $\Delta \pi_q(X)(i,j)=0$, and then we obtain $2^{i-1}\pi_q(a)=0$ by Proposition \ref{prop5}. Since $2$ is invertible in $\Z/q\Z$, it follows that $\pi_q(a)=0$ and hence that $\pi_q(X)$ is the zero-sequence of length $m$ in $\Z/q\Z$, in contradiction with the fact that $\pi_q(X)$ is balanced.
\end{proof}

We continue by studying arithmetic progressions with invertible common differences.

For every odd number $n$, we denote by \textit{$\alpha(n)$ the multiplicative order of $2^n$ modulo $n$}, i.e. the smallest positive integer $e$ such that $2^{en} \equiv 1 \pmod n$, namely
$$
\alpha(n) = \min\left\{ e\in\N^* \ \middle| \ 2^{en}\equiv 1\pmod n \right\}.
$$
\par For every positive integer $n$, we denote by $\varphi(n)$ \textit{the totient of $n$}, i.e. the number of positive integers less than or equal to $n$ that are coprime to $n$. Note that, for $n$ odd, the integer $\alpha(n)$ divides $\varphi(n)$.

In contrast with Theorem \ref{thm3}, the following result states that, for each $a$ and $d$ in $\Zn$ with $d$ invertible, there are infinitely many lengths $m$ for which the arithmetic progression $AP(a,d,m)$ is balanced.

\begin{thm}\label{mainthm1}
Let $n$ be an odd number. Let $a$ and $d$ be in $\Zn$ with $d$ invertible. Then, the arithmetic progression $AP(a,d,m)$ is balanced for every positive integer $m\equiv0$ or $-1\pmod{\alpha(n)n}$.
\end{thm}

This theorem will be proved at the end of this section.

The positive integer $\alpha(n)$ seems to be difficult to determine. Indeed, there is no general formula known to compute the multiplicative order of an integer modulo $n$ but, however, we get the following helpful propositions.

For every positive integer $n$, \textit{the radical of $n$}, denoted by $\rad(n)$, is the product of the distinct prime factors of $n$, that is,
$$
\rad(n) = \prod_{\stackrel{p\in\mathcal{P}}{p|n}}p.
$$
The radical of $n$ is also the largest square-free divisor of $n$.

\begin{prop}\label{prop2}
Let $n$ be an odd number. Then $\alpha(n)$ divides $\alpha(\rad(n))$.
\end{prop}

\begin{proof}
Let $p$ be a prime factor of $n$ such that $p^2$ divides $n$. We shall show that $\alpha(n)$ divides $\alpha(\frac{n}{p})$. There exists a positive integer $u$ such that
$$
2^{\alpha(\frac{n}{p})\frac{n}{p}} = 1+u\frac{n}{p}.
$$
It follows from the binomial theorem that
$$
2^{\alpha(\frac{n}{p})n} = {\left(2^{\alpha(\frac{n}{p})\frac{n}{p}}\right)}^{p} = {\left(1+u\frac{n}{p}\right)}^{p} = 1+\sum_{k=1}^{p-1}{\binom{p}{k}u^k\left(\frac{n}{p}\right)^k} + u^p\left(\frac{n}{p}\right)^p \equiv 1 \pmod{n},
$$
and so $\alpha(n)$ divides $\alpha(\frac{n}{p})$. We conclude by induction that $\alpha(n)$ divides $\alpha(\rad(n))$.
\end{proof}

\begin{prop}\label{prop8}
Let $p$ be an odd prime number. Then,
$$
\alpha(p^k)=\alpha(p),
$$
for every positive integer $k$.
\end{prop}

\begin{proof}
By Proposition \ref{prop2}, the integer $\alpha(p^k)$ divides $\alpha(p)$. It remains to prove that $\alpha(p)$ divides $\alpha(p^k)$. The congruence
$$
2^{\alpha(p^k)p^k} \equiv 1 \pmod{p^k}
$$
implies that
$$
2^{\alpha(p^k)p^k} \equiv 1 \pmod{p},
$$
and hence, by Fermat's little theorem, it follows that
$$
2^{\alpha(p^k)p} \equiv 2^{\alpha(p^k)p^k} \equiv 1 \pmod{p}.
$$
Therefore $\alpha(p)$ divides $\alpha(p^k)$. This completes the proof.
\end{proof}

\begin{prop}\label{prop3}
Let $n_1$ and $n_2$ be two relatively prime odd numbers. Then, $\alpha(n_1n_2)$ divides $\lcm(\alpha(n_1),\alpha(n_2))$.
\end{prop}

\begin{proof}
Let $i\in\{1,2\}$. The congruences
$$
2^{\alpha(n_i)n_i}\equiv1\pmod{n_i}
$$
imply that
$$
2^{\lcm(\alpha(n_1),\alpha(n_2))n_1n_2}\equiv1\pmod{n_i}.
$$
The result follows by the Chinese remainder theorem.
\end{proof}

For example, for $n_1=5$ and $n_2=3$, we have the equality $\alpha(15)=4=\lcm(4,2)=\lcm(\alpha(5),\alpha(3))$. However, $\alpha(n_1n_2)$ may be a strict factor of $\lcm(\alpha(n_1),\alpha(n_2))$, e.g. for $n=21$: $\alpha(21)=2$ and $\lcm(\alpha(7),\alpha(3))=\lcm(3,2)=6$. The table in Figure~\ref{fig3} shows the first values of $\alpha(n)$ for $n$ odd.

\begin{figure}[!h]
\centering
\begin{tabular}{cccc}
\begin{tabular}{|c|c|c|}
\hline
$n$ & $\rad(n)$ & $\alpha(n)$\\
\hline
\hline
$1$ & $1$ & $1$\\
\hline
$3$ & $3$ & $2$\\
\hline
$5$ & $5$ & $4$\\
\hline
$7$ & $7$ & $3$\\
\hline
$9$ & $3$ & $2$\\
\hline
$11$ & $11$ & $10$\\
\hline
$13$ & $13$ & $12$\\
\hline
$15$ & $5\cdot3$ & $4$\\
\hline
$17$ & $17$ & $8$\\
\hline
$19$ & $19$ & $18$\\
\hline
$21$ & $7\cdot3$ & $2$\\
\hline
$23$ & $23$ & $11$\\
\hline
$25$ & $5$ & $4$\\
\hline
\end{tabular} &
\begin{tabular}{|c|c|c|}
\hline
$n$ & $\rad(n)$ & $\alpha(n)$\\
\hline
\hline
$27$ & $3$ & $2$\\
\hline
$29$ & $29$ & $28$\\
\hline
$31$ & $31$ & $5$\\
\hline
$33$ & $11\cdot3$ & $10$\\
\hline
$35$ & $7\cdot5$ & $12$\\
\hline
$37$ & $37$ & $36$\\
\hline
$39$ & $13\cdot3$ & $4$\\
\hline
$41$ & $41$ & $20$\\
\hline
$43$ & $43$ & $14$\\
\hline
$45$ & $5\cdot3$ & $4$\\
\hline
$47$ & $47$ & $23$\\
\hline
$49$ & $7$ & $3$\\
\hline
$51$ & $17\cdot3$ & $8$\\
\hline
\end{tabular} &
\begin{tabular}{|c|c|c|}
\hline
$n$ & $\rad(n)$ & $\alpha(n)$\\
\hline
\hline
$53$ & $53$ & $52$\\
\hline
$55$ & $11\cdot5$ & $4$\\
\hline
$57$ & $19\cdot3$ & $6$\\
\hline
$59$ & $59$ & $58$\\
\hline
$61$ & $61$ & $60$\\
\hline
$63$ & $7\cdot3$ & $2$\\
\hline
$65$ & $13\cdot5$ & $12$\\
\hline
$67$ & $67$ & $66$\\
\hline
$69$ & $23\cdot3$ & $22$\\
\hline
$71$ & $71$ & $35$\\
\hline
$73$ & $73$ & $9$\\
\hline
$75$ & $5\cdot3$ & $4$\\
\hline
$77$ & $11\cdot7$ & $30$\\
\hline
\end{tabular} &
\begin{tabular}{|c|c|c|}
\hline
$n$ & $\rad(n)$ & $\alpha(n)$\\
\hline
\hline
$79$ & $79$ & $39$\\
\hline
$81$ & $3$ & $2$\\
\hline
$83$ & $83$ & $82$\\
\hline
$85$ & $17\cdot5$ & $8$\\
\hline
$87$ & $29\cdot3$ & $28$\\
\hline
$89$ & $89$ & $11$\\
\hline
$91$ & $13\cdot7$ & $12$\\
\hline
$93$ & $31\cdot3$ & $10$\\
\hline
$95$ & $19\cdot5$ & $36$\\
\hline
$97$ & $97$ & $48$\\
\hline
$99$ & $11\cdot3$ & $10$\\
\hline
$101$ & $101$ & $100$\\
\hline
$103$ & $103$ & $51$\\
\hline
\end{tabular}
\end{tabular}
\caption{\label{fig3}The first values of $\alpha(n)$ for $n$ odd}
\end{figure}

We end this section by proving Theorem \ref{mainthm1}, using the following two lemmas.

\begin{lem}\label{lemma3}
Let $n$ be a positive integer. Let $AP(a,d,m)=(x_1,x_2,\ldots,x_m)$ be an arithmetic progression beginning with $a\in\Zn$ and with invertible common difference $d\in\Zn$. Then, every $n$ consecutive terms of $AP(a,d,m)$ are distinct. In other words, for every $1\leq i\leq m-n+1$, we have
$$
\left\{x_i,x_{i+1},\ldots,x_{i+n-1}\right\}=\Zn.
$$
\end{lem}

\begin{proof}
Since the common difference $d$ is invertible in $\Zn$, it follows that, for every positive integers $i_1$ and $i_2$, we have
$$
x_{i_1}=x_{i_2}\ \Longleftrightarrow\ a+(i_1-1)d=a+(i_2-1)d\ \Longleftrightarrow\ (i_1-1)d=(i_2-1)d\ \Longleftrightarrow\ i_1\equiv i_2\pmod{n}.
$$
This completes the proof.
\end{proof}

\begin{lem}\label{lemma1}
Let $n$ be an odd number and $k$ a positive integer. Let $a$ and $d$ be in $\Zn$ with $d$ invertible. Then, the arithmetic progression $AP(a,d,k\alpha(n)n)$ is balanced if, and only if, its initial segment $AP(a,d,\alpha(n)n)$ is also balanced.
\end{lem}

\begin{proof}
We shall show that there exists a relationship between the multiplicity function of the Steinhaus triangle $\Delta AP(a,d,k\alpha(n)n)$ and that of $\Delta AP(a,d,\alpha(n)n)$. We set
$$
X=AP(a,d,k\alpha(n)n).
$$
We now consider the structure of the Steinhaus triangle $\Delta X$ depicted in Figure~\ref{fig4}. Recall that $\Delta X(i,j)$ denotes the $j$th element of the $i$th row of $\Delta X$, for every integer $1\leq i\leq k\alpha(n)n$ and every integer $1\leq j\leq k\alpha(n)n-i+1$.
\begin{figure}[!h]
\centering
\begin{pspicture}(4,3.964)
\pspolygon(0,3.464)(4,3.464)(2,0)
\psline(1,3.464)(0.5,2.598)(3.5,2.598)
\psline{<->}(0,3.664)(1,3.664)
\rput(0.5,3.864){\tiny{$\alpha(n)n$}}
\psline{<->}(1,3.664)(4,3.664)
\rput(2.5,3.864){\tiny{$(k-1)\alpha(n)n$}}
\psline{<->}(4.2,3.464)(4.2,2.598)
\rput(4.7,3.031){\tiny{$\alpha(n)n$}}
\psline{<->}(4.2,2.598)(4.2,0)
\rput(5.1,1.299){\tiny{$(k-1)\alpha(n)n$}}
\rput(0.5,3.175){$A$}
\rput(2.25,3.031){$C$}
\rput(2,1.732){$B$}
\end{pspicture}
\caption{\label{fig4}Structure of $\Delta X$}
\end{figure}
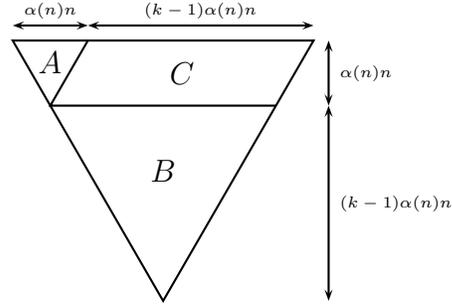
The subtriangle $A$ is defined by
$$
A = \left\{ \Delta X(i,j)\ \middle|\ 1\leq i\leq \alpha(n)n\ ,\ 1\leq j\leq \alpha(n)n-i+1 \right\}.
$$
Then $A$ is the Steinhaus triangle generated by the initial segment $AP(a,d,\alpha(n)n)$ of the sequence $X$, that is,
$$
A = \Delta AP(a,d,\alpha(n)n).
$$
The subtriangle $B$ is defined by
$$
B = \left\{ \Delta X(i,j)\ \middle|\ \alpha(n)n+1\leq i\leq k\alpha(n)n\ ,\ 1\leq j\leq k\alpha(n)n-i+1 \right\}.
$$
Then $B$ is the Steinhaus triangle generated by the derived sequence $\partial^{\alpha(n)n}X$, that is,
$$
B = \Delta\partial^{\alpha(n)n}X.
$$
Applying Proposition \ref{prop4}, we obtain that
$$
\partial^{\alpha(n)n}AP\left(a,d,k\alpha(n)n\right) = AP\left(2^{\alpha(n)n}a+2^{\alpha(n)n-1}\alpha(n)nd,2^{\alpha(n)n}d,(k-1)\alpha(n)n\right).
$$
Since $2^{\alpha(n)n}=1$, it immediately follows that
$$
B = \Delta AP(a,d,(k-1)\alpha(n)n).
$$
Finally, the multiset $C$ is defined by
$$
C = \left\{ \Delta X(i,j)\ \middle|\ 1\leq i\leq \alpha(n)n\ ,\ \alpha(n)n-i+2\leq j\leq k\alpha(n)n-i+1 \right\}.
$$
Then each row of $C$ is composed of $(k-1)\alpha(n)n$ consecutive terms of a derived sequence of $X$. Since, for every $0\leq i\leq k\alpha(n)n-1$, the derived sequence $\partial^iX$ of $X$ is an arithmetic progression with invertible common difference $2^{i}d$ by Proposition \ref{prop4}, it follows from Lemma \ref{lemma3} that each element of $\Zn$ occurs $(k-1)\alpha(n)$ times in each row of $C$. Therefore, the multiplicity function of $C$ is the constant function defined by
$$
\m_C(x) = (k-1){\alpha(n)}^2n,\ \forall x\in\Zn.
$$
Combining these results on the multisets $A$, $B$ and $C$, we have
$$
\m_{\Delta AP(a,d,k\alpha(n)n)}(x)
\begin{array}[t]{l}
= \m_A(x)+\m_B(x)+\m_C(x) \\
= \m_{\Delta AP(a,d,\alpha(n)n)}(x) + \m_{\Delta AP(a,d,(k-1)\alpha_nn)}(x) + (k-1){\alpha(n)}^2n,
\end{array}
$$
for every $x$ in $\Zn$. Thus, by induction on $k$, we obtain
$$
\m_{\Delta AP(a,d,k\alpha(n)n)}(x)=k\cdot\m_{\Delta AP(a,d,\alpha(n)n)}(x)+\binom{k}{2}{\alpha(n)}^2n,\ \forall x\in\Zn.
$$
This completes the proof.
\end{proof}

We are now ready to prove our main theorem.

\begin{proof}[Proof of Theorem \ref{mainthm1}]\ 
\par\textbf{1st Case: $\mathbf{m\equiv-1\pmod{\alpha(n)n}}$.}\\
We first derive the case $m\equiv -1 \pmod{\alpha(n)n}$ from the case $m\equiv 0\pmod{\alpha(n)n}$. Let $k$ be a positive integer and
$$
X=AP\left(a,d,k\alpha(n)n-1\right).
$$
By Proposition \ref{prop1}, the arithmetic progression
$$
Y=AP\left(2^{-1}a-2^{-2}d,2^{-1}d,k\alpha(n)n\right)
$$
is a primitive sequence of $X$. Since $Y$ is an arithmetic progression with invertible common difference $2^{-1}d$ and of length $k\alpha(n)n$, it follows from Lemma \ref{lemma3} that each element of $\Zn$ occurs $k\alpha(n)$ times in the sequence $Y$. Since $X$ is the derived sequence of $Y$, we have
$$
\m_{\Delta X}(x) = \m_{\Delta\partial Y}(x) = \m_{\Delta Y}(x) - \m_Y(x) = \m_{\Delta Y}(x) - k\alpha(n),
$$
for all $x$ in $\Zn$. Therefore, $X$ is balanced if and only if the sequence $Y$ is balanced. This completes the proof of the case $m\equiv -1\pmod{\alpha(n)n}$ from the case $m\equiv0\pmod{\alpha(n)n}$.

\par\textbf{2nd Case: $\mathbf{m\equiv0\pmod{\alpha(n)n}}$.}\\
We shall prove this case by induction on $n$. For $n=1$, it is clear that all finite sequences in $\Zn=\{0\}$ are balanced and so, the assertion is true for $n=1$. Let now $n>1$ be a positive integer and $p$ be the greatest prime factor of $n$. Suppose that the statement is true for $q=\frac{n}{p}$, i.e. every arithmetic progression with invertible common difference and of length $m\equiv0\pmod{\alpha(q)q}$ in $\Z/q\Z$ is balanced. Let $a$ and $d$ be in $\Zn$ with $d$ invertible. We will show that $AP(a,d,m)$ is balanced for every positive integer $m\equiv0\pmod{\alpha(n)n}$. By Lemma \ref{lemma1}, it is sufficient to prove that $AP(a,d,m)$ is balanced for one length $m$ multiple of $\alpha(n)n$.

We set
$$
\lambda=\varphi\left(\frac{\rad(n)}{p}\right).
$$
Then the integer $\lambda\alpha(p)$ is a multiple of $\alpha(n)$. Indeed, the integer $\alpha(n)$ divides $\alpha(\rad(n))$ by Proposition~\ref{prop2}, which divides $\alpha\left(\frac{\rad(n)}{p}\right)\alpha(p)$ by Proposition~\ref{prop3}, which divides $\varphi\left(\frac{\rad(n)}{p}\right)\alpha(p)$ by definition of the function $\alpha$.

We will prove that the sequence $X=AP(a,d,\lambda\alpha(p)n)$ is balanced. We begin by showing that the multiplicity function of $\Delta X$ is constant on each coset of the subgroup $q\Zn$. We consider the structure of the Steinhaus triangle $\Delta X$ depicted in Figure~\ref{fig5} where $\Delta X$ is constituted by the multisets $A_r$, $B_{(s,t)}$ and $C_u$. We shall show that $\mathbf{(1)}$ the multiplicity function $\m_{C_u}$ is constant for each $C_u$, $\mathbf{(2)}$ the multiplicity function of the union of the $B_{(s,t)}$ is constant, and $\mathbf{(3)}$ the multiplicity function of the union of the $A_r$ is constant on each coset of the subgroup $q\Zn$.

\begin{figure}[!p]
\centering
\begin{pspicture}(16.2,14.0296)

\pspolygon(0,14.0296)(16.2,14.0296)(13.5,9.3531)(2.7,9.3531)

\psline(2.7,14.0296)(1.35,11.6913)(14.85,11.6913)
\psline(4.05,11.6913)(2.7,9.3531)

\psline[linestyle=dashed](5.4,9.3531)(4.05,7.0148)(12.15,7.0148)
\psline[linestyle=dashed](6.75,7.0148)(5.4,4.6765)
\psline[linestyle=dashed](2.7,9.3531)(5.4,4.6765)
\psline[linestyle=dashed](13.5,9.3531)(10.8,4.6765)

\pspolygon(5.4,4.6765)(10.8,4.6765)(8.1,0)
\psline(8.1,4.6765)(6.75,2.3383)(9.45,2.3383)

\rput(8.775,12.8605){$C_1$}
\rput(8.775,10.5222){$C_2$}
\rput(8.775,3.5074){$C_{\alpha(p)-1}$}

\psline(0.45,14.0296)(0.225,13.6399)(2.475,13.6399)
\psline(0.675,13.6399)(0.45,13.2502)(2.25,13.2502)
\psline[linestyle=dashed,linewidth=0.01,dash=0.05](0.9,13.2502)(0.675,12.8605)(2.025,12.8605)
\psline[linestyle=dashed,linewidth=0.01,dash=0.05](1.125,12.8605)(0.9,12.4708)
\psline(0.9,12.4708)(1.8,12.4708)
\psline(1.35,12.4708)(1.125,12.0811)(1.575,12.0811)

\rput(1.4625,13.8348){\tiny{$B_{(1,1)}$}}
\rput(1.4625,13.4450){\tiny{$B_{(1,2)}$}}
\rput[bl](1.575,12.0811){\tiny{$\ \ \longleftarrow B_{(1,p-1)}$}}
\rput[br](0.225,13.6399){\tiny{$A_1\longrightarrow\ \ $}}
\rput[br](0.45,13.2502){\tiny{$A_2\longrightarrow\ \ $}}
\rput[br](1.125,12.0811){\tiny{$A_{p-1}\longrightarrow\ \ $}}
\rput[br](1.35,11.6913){\tiny{$A_{p}\longrightarrow\ \ $}}

\psline(1.8,11.6913)(1.575,11.3016)(3.825,11.3016)
\psline(2.025,11.3016)(1.8,10.9119)(3.6,10.9119)
\psline[linestyle=dashed,linewidth=0.01,dash=0.05](2.25,10.9119)(2.025,10.5222)(3.375,10.5222)
\psline[linestyle=dashed,linewidth=0.01,dash=0.05](2.475,10.5222)(2.25,10.1325)
\psline(2.25,10.1325)(3.15,10.1325)
\psline(2.7,10.1325)(2.475,9.7428)(2.925,9.7428)

\rput(2.8125,11.4965){\tiny{$B_{(2,1)}$}}
\rput(2.8125,11.1068){\tiny{$B_{(2,2)}$}}
\rput[bl](2.925,9.7428){\tiny{$\ \ \longleftarrow B_{(2,p-1)}$}}
\rput[br](1.575,11.3016){\tiny{$A_{p+1}\longrightarrow\ \ $}}
\rput[br](1.8,10.9119){\tiny{$A_{p+2}\longrightarrow\ \ $}}
\rput[br](2.475,9.7428){\tiny{$A_{2p-1}\longrightarrow\ \ $}}
\rput[br](2.7,9.3531){\tiny{$A_{2p}\longrightarrow\ \ $}}

\psline[linestyle=dashed,linewidth=0.01,dash=0.05](3.15,9.3531)(2.925,8.9634)(5.175,8.9634)
\psline[linestyle=dashed,linewidth=0.01,dash=0.05](3.375,8.9634)(3.15,8.5737)(4.95,8.5737)
\psline[linestyle=dashed,linewidth=0.01,dash=0.05](3.6,8.5737)(3.375,8.1839)(4.725,8.1839)
\psline[linestyle=dashed,linewidth=0.01,dash=0.05](3.825,8.1839)(3.6,7.7942)(4.5,7.7942)
\psline[linestyle=dashed,linewidth=0.01,dash=0.05](4.05,7.7942)(3.825,7.4045)(4.275,7.4045)

\psline[linestyle=dashed,linewidth=0.01,dash=0.05](4.5,7.0148)(4.275,6.6251)(6.525,6.6251)
\psline[linestyle=dashed,linewidth=0.01,dash=0.05](4.725,6.6251)(4.5,6.2354)(6.3,6.2354)
\psline[linestyle=dashed,linewidth=0.01,dash=0.05](4.95,6.2354)(4.725,5.8457)(6.075,5.8457)
\psline[linestyle=dashed,linewidth=0.01,dash=0.05](5.175,5.8457)(4.95,5.4560)(5.85,5.4560)
\psline[linestyle=dashed,linewidth=0.01,dash=0.05](5.4,5.4560)(5.175,5.0662)(5.625,5.0662)

\psline(5.85,4.6765)(5.625,4.2868)(7.875,4.2868)
\psline(6.075,4.2868)(5.85,3.8971)(7.65,3.8971)
\psline[linestyle=dashed,linewidth=0.01,dash=0.05](6.3,3.8971)(6.075,3.5074)(7.425,3.5074)
\psline[linestyle=dashed,linewidth=0.01,dash=0.05](6.525,3.5074)(6.3,3.1177)
\psline(6.3,3.1177)(7.2,3.1177)
\psline(6.75,3.1177)(6.525,2.7280)(6.975,2.7280)

\rput(6.8625,4.4817){\tiny{$B_{(\alpha(p)-1,1)}$}}
\rput(6.8625,4.0920){\tiny{$B_{(\alpha(p)-1,2)}$}}
\rput[bl](6.975,2.7280){\tiny{$\ \ \longleftarrow B_{(\alpha(p)-1,p-1)}$}}
\rput[br](5.625,4.2868){\tiny{$A_{(\alpha(p)-2)p+1}\longrightarrow\ \ $}}
\rput[br](5.85,3.8971){\tiny{$A_{(\alpha(p)-2)p+2}\longrightarrow\ \ $}}
\rput[br](6.525,2.7280){\tiny{$A_{(\alpha(p)-1)p-1}\longrightarrow\ \ $}}
\rput[br](6.75,2.3383){\tiny{$A_{(\alpha(p)-1)p}\longrightarrow\ \ $}}

\psline(7.2,2.3383)(6.975,1.9486)(9.225,1.9486)
\psline(7.425,1.9486)(7.2,1.5588)(9,1.5588)
\psline[linestyle=dashed,linewidth=0.01,dash=0.05](7.65,1.5588)(7.425,1.1691)(8.775,1.1691)
\psline[linestyle=dashed,linewidth=0.01,dash=0.05](7.875,1.1691)(7.65,0.7794)
\psline(7.65,0.7794)(8.55,0.7794)
\psline(8.1,0.7794)(7.875,0.3897)(8.325,0.3897)

\rput(8.2125,2.1434){\tiny{$B_{(\alpha(p),1)}$}}
\rput(8.2125,1.7537){\tiny{$B_{(\alpha(p),2)}$}}
\rput[bl](8.325,0.3897){\tiny{$\ \ \longleftarrow B_{(\alpha(p),p-1)}$}}
\rput[br](6.975,1.9486){\tiny{$A_{(\alpha(p)-1)p+1}\longrightarrow\ \ $}}
\rput[br](7.2,1.5588){\tiny{$A_{(\alpha(p)-1)p+2}\longrightarrow\ \ $}}
\rput[br](7.875,0.3897){\tiny{$A_{\alpha(p)p-1}\longrightarrow\ \ $}}
\rput[br](8.1,0){\tiny{$A_{\alpha(p)p}\longrightarrow\ \ $}}

\end{pspicture}
\caption{\label{fig5}Structure of $\Delta X$}
\end{figure}
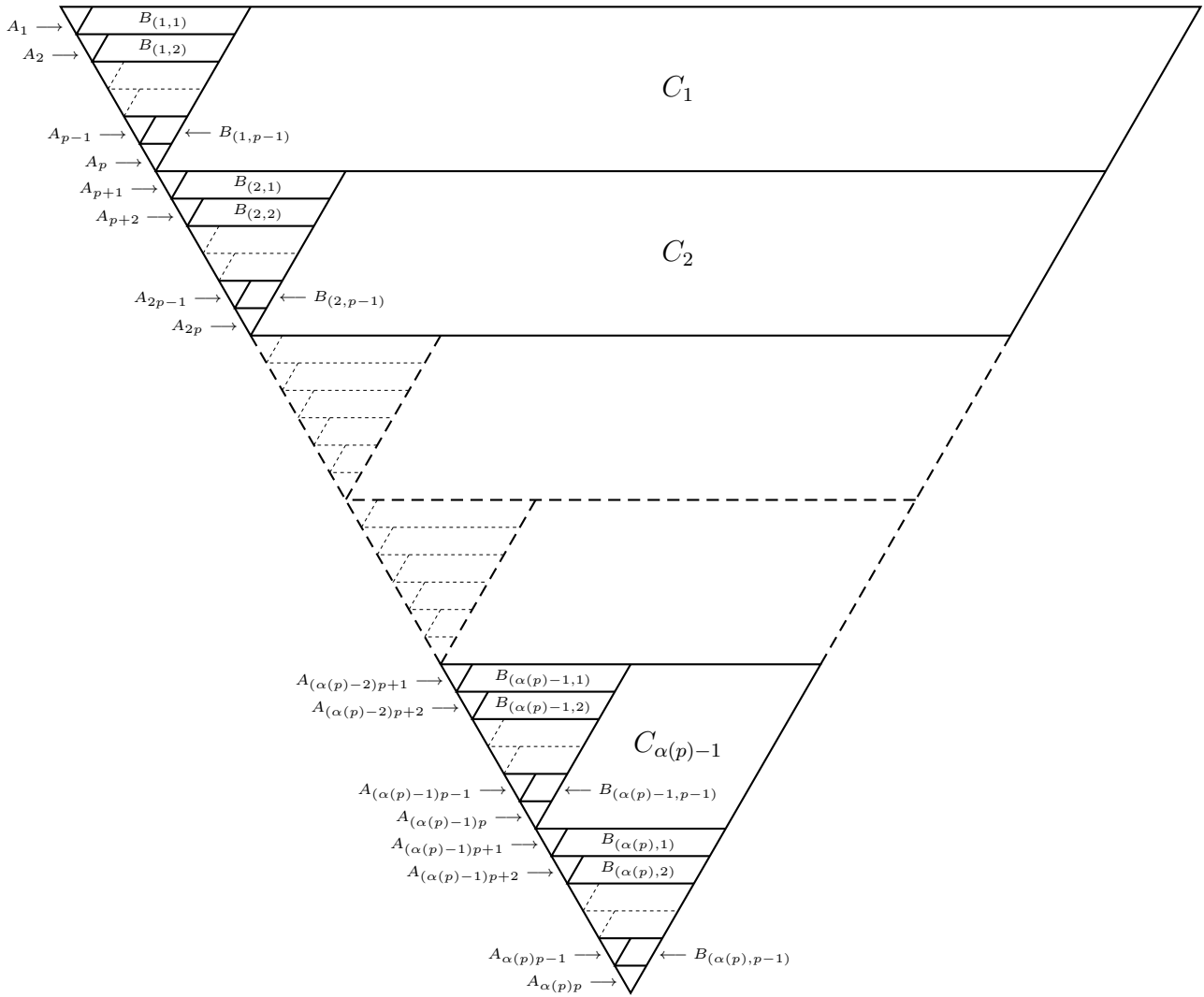

\noindent\textbf{Step (1):} The multiplicity function $\m_{C_u}$ is constant for every $1\leq u\leq\alpha(p)-1$.
\par For every integer $1\leq u\leq \alpha(p)-1$, the multiset $C_u$ is defined by
$$
C_u = \left\{ {\Delta X}(i,j)\ \middle|\ (u-1)\lambda n+1\leq i\leq u\lambda n\ ,\ u\lambda n-i+2\leq j\leq\lambda\alpha(p)n-i+1 \right\},
$$
where $\Delta X(i,j)$ denotes the $j$th element in the $i$th row of $\Delta X$, for every integer $1\leq i\leq \lambda\alpha(p)n$ and every integer $1\leq j\leq \lambda\alpha(p)n-i+1$. As depicted in Figure~\ref{fig6}, each multiset $C_u$ is a parallelogram of $\lambda n$ rows and $(\alpha(p)-u)\lambda n$ columns.

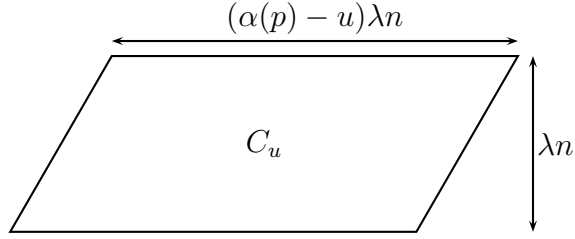
\begin{figure}[!h]
\centering
\begin{pspicture}(6.75,2.8383)
\pspolygon(0,0)(1.35,2.3383)(6.75,2.3383)(5.4,0)
\psline{<->}(1.35,2.5383)(6.75,2.5383)
\rput(4.05,2.8383){$(\alpha(p)-u)\lambda n$}
\psline{<->}(6.95,2.3383)(6.95,0)
\rput(7.25,1.1691){$\lambda n$}
\rput(3.375,1.1691){$C_u$}
\end{pspicture}
\caption{\label{fig6}Structure of $C_u$}
\end{figure}

Let $1\leq u\leq \alpha(p)-1$. Each row of $C_u$ is composed of $(\alpha(p)-u)\lambda n$ consecutive terms of a derived sequence of $X$. For every $0\leq i\leq\lambda\alpha(p)n-1$, the derived sequence $\partial^{i}X$ of $X$ is an arithmetic progression with invertible common difference $2^{i}d$ by Proposition \ref{prop4}. It follows from Lemma \ref{lemma3} that each element of $\Zn$ occurs $(\alpha(p)-u)\lambda$ times in each row of $C_u$. Therefore, the multiplicity function of $C_u$ is the constant function defined by
$$
\m_{C_u}(x) = (\alpha(p)-u){\lambda}^2n,\ \forall x\in\Zn.
$$

\noindent\textbf{Step (2):} The multiplicity function of the union of all the multisets $B_{(s,t)}$ is constant.
\par For every integer $1\leq s\leq\alpha(p)$ and every integer $1\leq t\leq p-1$, the multiset $B_{(s,t)}$ is defined by
$$
B_{(s,t)} = \left\{ {\Delta X}(i,j)\ \middle|\ \begin{array}{l}((s-1)p+t-1)\lambda\frac{n}{p}+1\leq i\leq((s-1)p+t)\lambda\frac{n}{p}\\ ((s-1)p+t)\lambda\frac{n}{p}-i+2\leq j\leq s\lambda n-i+1\end{array} \right\}.
$$
As depicted in Figure~\ref{fig7}, each multiset $B_{(s,t)}$ is a parallelogram of $\lambda\frac{n}{p}$ rows and $(p-t)\lambda\frac{n}{p}$ columns.

\begin{figure}[!h]
\centering
\begin{pspicture}(6.75,2.8383)
\pspolygon(0,0)(1.35,2.3383)(6.75,2.3383)(5.4,0)
\psline{<->}(1.35,2.5383)(6.75,2.5383)
\rput(4.05,2.8383){$(p-t)\lambda\frac{n}{p}$}
\psline{<->}(6.95,2.3383)(6.95,0)
\rput(7.25,1.1691){$\lambda\frac{n}{p}$}
\rput(3.375,1.1691){$B_{(s,t)}$}
\end{pspicture}
\caption{\label{fig7}Structure of $B_{(s,t)}$}
\end{figure}
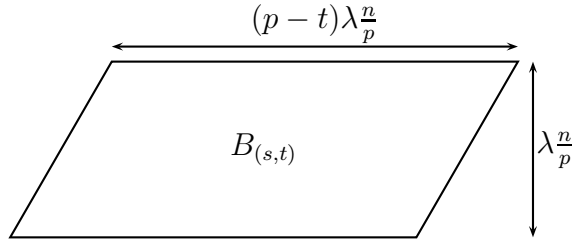

We will construct a fixed-point-free involution $\Psi$ on the set of pairs $(s,t)$ such that the multiplicity function of the multiset union $B_{(s,t)}\cup B_{\Psi(s,t)}$ is constant for every pair $(s,t)$. Let
$$
\Psi : \left\{
\begin{array}{ccc}
\llbracket1,\alpha(p)\rrbracket\times\llbracket1,p-1\rrbracket & \longrightarrow & \llbracket1,\alpha(p)\rrbracket\times\llbracket1,p-1\rrbracket\\
(s,t) & \longmapsto & \left(\psi(s,t),p-t\right)
\end{array}
,\right.
$$
where $\psi : \llbracket1,\alpha(p)\rrbracket\times\llbracket1,p-1\rrbracket\longrightarrow\llbracket1,\alpha(p)\rrbracket$ is the function which assigns to each pair $(s,t)$ the positive integer $\psi(s,t)$ in $\llbracket1,\alpha(p)\rrbracket$ which is equivalent to $s+2t-1$ modulo $\alpha(p)$, that is,
$$
\psi(s,t) \equiv s+2t-1 \pmod{\alpha(p)},\ \forall 1\leq s\leq\alpha(p),\ \forall 1\leq t\leq p-1.
$$
Since $\alpha(p)$ divides $\varphi(p)=p-1$, it follows that
$$
\psi(\psi(s,t),p-t)\equiv\psi(s,t)+2p-2t-1\equiv s+2(p-1) \equiv s \pmod{\alpha(p)}
$$
and hence, we obtain that
$$
\Psi\left(\Psi(s,t)\right)=\Psi(\psi(s,t),p-t)=(\psi(\psi(s,t),p-t),t)=(s,t),
$$
for every $(s,t)$ in $\llbracket1,\alpha(p)\rrbracket\times\llbracket1,p-1\rrbracket$. Moreover, this involution has no fixed point. Indeed, if $(s,t)$ were a fixed point of $\Psi$, then
$$
(s,t) = \Psi(s,t) = (\psi(s,t),p-t),
$$
implying $p=2t$, in contradiction with the parity of $p$. We have proved that $\Psi$ is a fixed-point-free involution on the set $\llbracket1,\alpha(p)\rrbracket\times\llbracket1,p-1\rrbracket$.

Let $1\leq s\leq\alpha(p)$ and let $1\leq t\leq p-1$. If we denote by $B_{(s,t)}^{(v)}$ the $v$th row of $B_{(s,t)}$, that is,
\small
$$
B_{(s,t)}^{(v)} = \left\{ \Delta X\left( ((s-1)p+t-1)\lambda\frac{n}{p}+v,j \right)\ \middle|\ \lambda\frac{n}{p}-v+2\leq j\leq (p-t+1)\lambda\frac{n}{p}-v+1 \right\},
$$
\normalsize
for all $1\leq v\leq \lambda\frac{n}{p}$, then
$$
B_{(s,t)} = \bigcup_{v=1}^{\lambda\frac{n}{p}}{B_{(s,t)}^{(v)}}.
$$
Let $1\leq v\leq\lambda\frac{n}{p}$. The sequence $B_{(s,t)}^{(v)}$ is composed of $(p-t)\lambda\frac{n}{p}$ consecutive terms of the derived sequence
$$
\partial^{((s-1)p+t-1)\lambda\frac{n}{p}+v-1}X,
$$
which is an arithmetic progression with common difference
$$
2^{((s-1)p+t-1)\lambda\frac{n}{p}+v-1}d
$$
by Proposition \ref{prop4}. It follows that
$$
B_{(s,t)}^{(v)} = AP\left( b_{(s,t)}^{(v)}\ ,\ 2^{((s-1)p+t-1)\lambda\frac{n}{p}+v-1}d\ ,\ (p-t)\lambda\frac{n}{p} \right),
$$
with
$$
b_{(s,t)}^{(v)} = \Delta X\left(((s-1)p+t-1)\lambda\frac{n}{p}+v,\lambda\frac{n}{p}-v+2\right).
$$
We will show that the sequence $B_{(s,t)}^{(v)}\circ B_{\Psi(s,t)}^{(v)}$, \textit{the concatenation of the sequences $\mathit{B_{(s,t)}^{(v)}}$ and $\mathit{B_{\Psi(s,t)}^{(v)}}$}, is an arithmetic progression with invertible common difference and of length $\lambda n$. The congruence $p\equiv1\pmod{\alpha(p)}$ implies that
$$
(s-1)p+t-1 \equiv s+t-2 \pmod{\alpha(p)},
$$
and
$$
(\psi(s,t)-1)p+(p-t)-1 \equiv \psi(s,t)-1-t \equiv s+t-2 \pmod{\alpha(p)}.
$$
Since $\alpha(n)$ divides $\lambda\alpha(p)$, it follows that
$$
2^{((\psi(s,t)-1)p+(p-t)-1)\lambda} \equiv 2^{((s-1)p+t-1)\lambda} \equiv 2^{(s+t-2)\lambda} \pmod{n},
$$
and hence,
$$
2^{((\psi(s,t)-1)p+(p-t)-1)\lambda\frac{n}{p}+v-1}d = 2^{((s-1)p+t-1)\lambda\frac{n}{p}+v-1}d = 2^{(s+t-2)\lambda\frac{n}{p}+v-1}d.
$$
Therefore the sequences $B_{(s,t)}^{(v)}$ and $B_{\Psi(s,t)}^{(v)}$ are both arithmetic progressions with common difference
$$
2^{(s+t-2)\lambda\frac{n}{p}+v-1}d.
$$
It remains to prove that $b_{\Psi(s,t)}^{(v)}$ can be expressed as the next element of the arithmetic progression $B_{(s,t)}^{(v)}$. Since
$$
b_{\Psi(s,t)}^{(v)}
\begin{array}[t]{l}
= \Delta X\left(((\Psi(s,t)-1)p+(p-t)-1)\lambda\frac{n}{p}+v,\lambda\frac{n}{p}-v+2\right)\\
= 2^{((\psi(s,t)-1)p+(p-t)-1)\lambda\frac{n}{p}+v-2}\left( 2a + \left( 2\left( \lambda\frac{n}{p}-v+2 \right) + \right.\right.\\
\ \ \ + \left.\left.\left( ((\psi(s,t)-1)p+(p-t)-1)\lambda\frac{n}{p}+v \right) -3 \right)d \right)\\
= 2^{(s+t-2)\lambda\frac{n}{p}+v-2}\left( 2a + \left( ((p-t)+1)\lambda\frac{n}{p}-v+1 \right) d \right)\\
= 2^{(s+t-2)\lambda\frac{n}{p}+v-2}\left( 2a + \left( (t+1)\lambda\frac{n}{p}-v+1 \right) d \right) + (p-2t)\lambda\frac{n}{p}\left(2^{(s+t-2)\lambda\frac{n}{p}+v-2}d\right)\\
= b_{(s,t)}^{(v)} + (p-t)\lambda\frac{n}{p}\left(2^{(s+t-2)\lambda\frac{n}{p}+v-2}d\right),
\end{array}
$$
it follows that
$$
B_{(s,t)}^{(v)}\circ B_{\Psi(s,t)}^{(v)} = AP\left( b_{(s,t)}^{(v)}\ ,\ 2^{(s+t-2)\lambda\frac{n}{p}+v-1}d\ ,\ \lambda n \right),
$$
and so, each element of $\Zn$ occurs $\lambda$ times in $B_{(s,t)}^{(v)}\circ B_{\Psi(s,t)}^{(v)}$ for every $1\leq v\leq\lambda\frac{n}{p}$. Then the multiplicity function of the multiset union $B_{(s,t)}\cup B_{\Psi(s,t)}$ is the constant function defined by
$$
\m_{B_{(s,t)}\cup B_{\Psi(s,t)}}(x) = {\lambda}^2\frac{n}{p},\ \forall x\in\Zn.
$$
If we denote by $B$ the union of all the multisets $B_{(s,t)}$, then
$$
\m_B(x) = \displaystyle\sum_{s=1}^{\alpha(p)}\sum_{t=1}^{p-1}\m_{B_{(s,t)}}(x) = \frac{1}{2}\sum_{s=1}^{\alpha(p)}\sum_{t=1}^{p-1}\m_{B_{(s,t)}\cup B_{\Psi(s,t)}}(x) = \displaystyle\frac{1}{2}\sum_{s=1}^{\alpha(p)}\sum_{t=1}^{p-1}{\lambda}^2\frac{n}{p} = \alpha(p){\lambda}^{2}\frac{(p-1)n}{2p},
$$
for every $x$ in $\Zn$, since $\Psi$ is a fixed-point-free involution on $\llbracket1,\alpha(p)\rrbracket\times\llbracket1,p-1\rrbracket$.

\noindent\textbf{Step (3):} The multiplicity function of the union of all the multisets $A_r$ is constant on each coset of the subgroup $\frac{n}{p}\Zn$.
\par For every integer $1\leq r\leq\alpha(p)p$, the multiset $A_r$ is defined by
$$
A_r = \left\{ {\Delta X}(i,j)\ \middle|\ (r-1)\lambda\frac{n}{p}+1\leq i\leq r\lambda\frac{n}{p}\ ,\ 1\leq j\leq r\lambda\frac{n}{p}-i+1 \right\}.
$$
As depicted in Figure~\ref{fig8}, each multiset $A_r$ is a triangle associated to a sequence of length $\lambda\frac{n}{p}$.

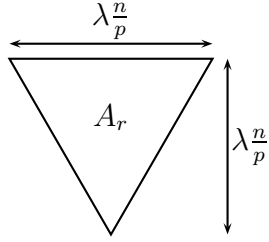
\begin{figure}[!h]
\centering
\begin{pspicture}(2.7,2.8383)
\pspolygon(0,2.3383)(2.7,2.3383)(1.35,0)
\psline{<->}(0,2.5383)(2.7,2.5383)
\rput(1.35,2.8383){$\lambda\frac{n}{p}$}
\psline{<->}(2.9,2.3383)(2.9,0)
\rput(3.2,1.1691){$\lambda\frac{n}{p}$}
\rput(1.35,1.5588){$A_r$}
\end{pspicture}
\caption{\label{fig8}Structure of $A_r$}
\end{figure}

If we denote by $X_r$ the sequence of the first $\lambda\frac{n}{p}$ terms of the derived sequence $\partial^{(r-1)\lambda\frac{n}{p}}X$, then $A_r$ is the Steinhaus triangle generated by $X_r$, for every $1\leq r\leq\alpha(p)p$. It is clear that there exists a correspondence between $A_r$ and the whole Steinhaus triangle $\Delta X$. Indeed, for every integer $1\leq i\leq\lambda\frac{n}{p}$ and every integer $1\leq j\leq\lambda\frac{n}{p}-i+1$, we have
$$
\Delta X_r(i,j) = \Delta X\left((r-1)\lambda\frac{n}{p}+i,j\right).
$$
Let $1\leq l\leq \alpha(p)$, $1\leq i\leq\lambda\frac{n}{p}$ and $1\leq j\leq \lambda\frac{n}{p}-i+1$. We will prove that each element of the coset
$$
\Delta X_l(i,j)+\frac{n}{p}\Zn
$$
occurs once in the multiset
$$
\left\{\Delta X_{l+k\alpha(p)}(i,j)\ \middle|\ k\in\llbracket0,p-1\rrbracket\right\}.
$$
First, the equality
$$
\lambda n - \lambda\frac{n}{p} = \lambda(p-1)\frac{n}{p} = \varphi\left(\frac{\rad(n)}{p}\right)(p-1)\frac{n}{p} = \varphi(\rad(n))\frac{n}{p} = \varphi(n)\frac{\rad(n)}{p}
$$
implies that
$$
2^{\lambda n} \equiv 2^{\lambda\frac{n}{p}} \pmod{n},
$$
and so,
$$
2^{\alpha(p)\lambda\frac{n}{p}} \equiv 2^{\alpha(p)\lambda n} \equiv 1 \pmod{n},
$$
since $\alpha(n)$ divides $\lambda\alpha(p)$. This leads to
$$
\Delta X_{l+k\alpha(p)}\left(i,j\right)
\begin{array}[t]{l}
= \Delta X\left( (k\alpha(p)+l-1)\lambda\frac{n}{p}+i,j \right)\\
= 2^{(k\alpha(p)+l-1)\lambda\frac{n}{p}+i-1}\left( 2a+\left( 2j+(k\alpha(p)+l-1)\lambda\frac{n}{p}+i-3 \right)d \right)\\
= 2^{k\alpha(p)\lambda\frac{n}{p}}2^{(l-1)\lambda\frac{n}{p}+i-1}\left( 2a+\left( 2j+(k\alpha(p)+l-1)\lambda\frac{n}{p}+i-3 \right)d \right)\\
= 2^{(l-1)\lambda\frac{n}{p}+i-1}\left( 2a+\left( 2j+(l-1)\lambda\frac{n}{p}+i-3 \right)d \right) + k\left(2^{(l-1)\lambda\frac{n}{p}+i-1}\lambda\alpha(p)d\right)\frac{n}{p}\\
= \Delta X\left((l-1)\lambda\frac{n}{p}+i,j\right) + k\left(2^{(l-1)\lambda\frac{n}{p}+i-1}\lambda\alpha(p)d\right)\frac{n}{p}\\
= \Delta X_l\left(i,j\right) + k\left(2^{(l-1)\lambda\frac{n}{p}+i-1}\lambda\alpha(p)d\right)\frac{n}{p},
\end{array}
$$
for every integer $0\leq k\leq p-1$. The congruence $p\equiv1\pmod{\alpha(p)}$ implies that $\alpha(p)$ is not divisible by $p$. Moreover, since $p$ is the greatest prime factor of $n$, it follows that $\lambda=\varphi\left(\frac{\rad(n)}{p}\right)$ is relatively prime to $p$ and hence, the integer $\lambda\alpha(p)$ is not divisible by $p$. Therefore, we obtain the following multiset equality
$$
\left\{ \Delta X_{l+k\alpha(p)}(i,j)\ \middle|\ k\in\llbracket0,p-1\rrbracket \right\} = \left\{ \Delta X_{l}(i,j),\Delta X_{l}(i,j)+\frac{n}{p},\ldots,\Delta X_{l}(i,j)+\frac{(p-1)n}{p}\right\},
$$
for every $1\leq l\leq\alpha(p)$, $1\leq i\leq\lambda\frac{n}{p}$ and $1\leq j\leq\lambda\frac{n}{p}-i+1$. If we denote by $A$ the union of all the multisets $A_r$, then the multiplicity function of $A$ is constant on each coset of the subgroup $\frac{n}{p}\Zn$.

We now combine the results obtained above. By Steps 1 and 2, we have
$$
\m_{\Delta X}(x)
\begin{array}[t]{l}
= \m_A(x) + \m_B(x) + \displaystyle\sum_{u=1}^{\alpha(p)-1}\m_{C_u}(x)\\
= \m_A(x) + \alpha(p){\lambda}^{2}\frac{(p-1)n}{2p} + \displaystyle\sum_{u=1}^{\alpha(p)-1}(\alpha(p)-u){\lambda}^{2}n\\
= \m_A(x) + \alpha(p){\lambda}^{2}\frac{(p-1)n}{2p} + \displaystyle\binom{\alpha(p)}{2}{\lambda}^{2}n,
\end{array}
$$
for every $x$ in $\Zn$ and so, by Step 3, the multiplicity function $\m_{\Delta X}$ is constant on each coset of the subgroup $\frac{n}{p}\Zn=q\Zn$.

We now end the proof by showing that $\pi_q(X)$, the image of the sequence $X$ under the surjective morphism $\pi_q : \Zn\twoheadlongrightarrow\Z/q\Z$ with $q=\frac{n}{p}$, is balanced. First, the sequence $\pi_q(X)$ is the arithmetic progression beginning with $\pi_q(a)\in\Z/q\Z$, with common difference $\pi_q(d)\in\Z/q\Z$ and of length $\lambda\alpha(p)n$, that is,
$$
\pi_q\left(X\right) = \pi_q\left(AP(a,d,\lambda\alpha(p)n)\right) = AP\left( \pi_q(a),\pi_q(d),\lambda\alpha(p)n \right).
$$
Moreover, the integer $\lambda\alpha(p)$ is divisible by $\alpha(q)$. Indeed, if $v_p(n)\geq2$, then $\rad(q)=\rad(n)$ and so $\alpha(q)$ divides $\alpha(\rad(q))=\alpha(\rad(n))$ by Proposition \ref{prop2}. As seen before, $\lambda\alpha(p)$ is divisible by $\alpha(\rad(n))$ and then, $\alpha(q)$ divides $\lambda\alpha(p)$. Otherwise, if $v_p(n)=1$, then $\rad(q)=\frac{\rad(n)}{p}$ and so $\alpha(q)$ divides $\alpha(\rad(q))=\alpha\left(\frac{\rad(n)}{p}\right)$ by Proposition \ref{prop2}. Since $\lambda=\varphi\left(\frac{\rad(n)}{p}\right)$ is divisible by $\alpha\left(\frac{\rad(n)}{p}\right)$, it follows that $\alpha(q)$ divides $\lambda$. In all cases, we have
$$
\lambda\alpha(p) \equiv 0 \pmod{\alpha(q)}.
$$
Therefore, the induction hypothesis implies that the sequence $\pi_q(X)$ is balanced, since it is an arithmetic progression with invertible common difference $\pi_q(d)$ and of length $\lambda\alpha(p)n$ divisible by $\alpha(q)q$.

We conclude that the sequence $X$ is balanced by Theorem \ref{thm2}. This completes the proof of Theorem \ref{mainthm1}.
\end{proof}

For example, in $\Z/7\Z$, the arithmetic progression $AP(1,3,20)$ is balanced since $\alpha(7)=3$ and $3$ is an invertible element in $\Z/7\Z$. Indeed, each element of $\Z/7\Z$ occurs $30$ times in this Steinhaus triangle.

\begin{figure}[!h]
\centering
\begin{pspicture}(11,8.8768)
\pspolygon(0.375,8.8768)(10.625,8.8768)(5.5,0)
\rput(0.75,8.6603){$1$}
\rput(1.25,8.6603){$4$}
\rput(1.75,8.6603){$0$}
\rput(2.25,8.6603){$3$}
\rput(2.75,8.6603){$6$}
\rput(3.25,8.6603){$2$}
\rput(3.75,8.6603){$5$}
\rput(4.25,8.6603){$1$}
\rput(4.75,8.6603){$4$}
\rput(5.25,8.6603){$0$}
\rput(5.75,8.6603){$3$}
\rput(6.25,8.6603){$6$}
\rput(6.75,8.6603){$2$}
\rput(7.25,8.6603){$5$}
\rput(7.75,8.6603){$1$}
\rput(8.25,8.6603){$4$}
\rput(8.75,8.6603){$0$}
\rput(9.25,8.6603){$3$}
\rput(9.75,8.6603){$6$}
\rput(10.25,8.6603){$2$}

\rput(1,8.2272){$5$}
\rput(1.5,8.2272){$4$}
\rput(2,8.2272){$3$}
\rput(2.5,8.2272){$2$}
\rput(3,8.2272){$1$}
\rput(3.5,8.2272){$0$}
\rput(4,8.2272){$6$}
\rput(4.5,8.2272){$5$}
\rput(5,8.2272){$4$}
\rput(5.5,8.2272){$3$}
\rput(6,8.2272){$2$}
\rput(6.5,8.2272){$1$}
\rput(7,8.2272){$0$}
\rput(7.5,8.2272){$6$}
\rput(8,8.2272){$5$}
\rput(8.5,8.2272){$4$}
\rput(9,8.2272){$3$}
\rput(9.5,8.2272){$2$}
\rput(10,8.2272){$1$}

\rput(1.25,7.7942){$2$}
\rput(1.75,7.7942){$0$}
\rput(2.25,7.7942){$5$}
\rput(2.75,7.7942){$3$}
\rput(3.25,7.7942){$1$}
\rput(3.75,7.7942){$6$}
\rput(4.25,7.7942){$4$}
\rput(4.75,7.7942){$2$}
\rput(5.25,7.7942){$0$}
\rput(5.75,7.7942){$5$}
\rput(6.25,7.7942){$3$}
\rput(6.75,7.7942){$1$}
\rput(7.25,7.7942){$6$}
\rput(7.75,7.7942){$4$}
\rput(8.25,7.7942){$2$}
\rput(8.75,7.7942){$0$}
\rput(9.25,7.7942){$5$}
\rput(9.75,7.7942){$3$}

\rput(1.5,7.3612){$2$}
\rput(2,7.3612){$5$}
\rput(2.5,7.3612){$1$}
\rput(3,7.3612){$4$}
\rput(3.5,7.3612){$0$}
\rput(4,7.3612){$3$}
\rput(4.5,7.3612){$6$}
\rput(5,7.3612){$2$}
\rput(5.5,7.3612){$5$}
\rput(6,7.3612){$1$}
\rput(6.5,7.3612){$4$}
\rput(7,7.3612){$0$}
\rput(7.5,7.3612){$3$}
\rput(8,7.3612){$6$}
\rput(8.5,7.3612){$2$}
\rput(9,7.3612){$5$}
\rput(9.5,7.3612){$1$}

\rput(1.75,6.9282){$0$}
\rput(2.25,6.9282){$6$}
\rput(2.75,6.9282){$5$}
\rput(3.25,6.9282){$4$}
\rput(3.75,6.9282){$3$}
\rput(4.25,6.9282){$2$}
\rput(4.75,6.9282){$1$}
\rput(5.25,6.9282){$0$}
\rput(5.75,6.9282){$6$}
\rput(6.25,6.9282){$5$}
\rput(6.75,6.9282){$4$}
\rput(7.25,6.9282){$3$}
\rput(7.75,6.9282){$2$}
\rput(8.25,6.9282){$1$}
\rput(8.75,6.9282){$0$}
\rput(9.25,6.9282){$6$}

\rput(2,6.4952){$6$}
\rput(2.5,6.4952){$4$}
\rput(3,6.4952){$2$}
\rput(3.5,6.4952){$0$}
\rput(4,6.4952){$5$}
\rput(4.5,6.4952){$3$}
\rput(5,6.4952){$1$}
\rput(5.5,6.4952){$6$}
\rput(6,6.4952){$4$}
\rput(6.5,6.4952){$2$}
\rput(7,6.4952){$0$}
\rput(7.5,6.4952){$5$}
\rput(8,6.4952){$3$}
\rput(8.5,6.4952){$1$}
\rput(9,6.4952){$6$}

\rput(2.25,6.0622){$3$}
\rput(2.75,6.0622){$6$}
\rput(3.25,6.0622){$2$}
\rput(3.75,6.0622){$5$}
\rput(4.25,6.0622){$1$}
\rput(4.75,6.0622){$4$}
\rput(5.25,6.0622){$0$}
\rput(5.75,6.0622){$3$}
\rput(6.25,6.0622){$6$}
\rput(6.75,6.0622){$2$}
\rput(7.25,6.0622){$5$}
\rput(7.75,6.0622){$1$}
\rput(8.25,6.0622){$4$}
\rput(8.75,6.0622){$0$}

\rput(2.5,5.6292){$2$}
\rput(3,5.6292){$1$}
\rput(3.5,5.6292){$0$}
\rput(4,5.6292){$6$}
\rput(4.5,5.6292){$5$}
\rput(5,5.6292){$4$}
\rput(5.5,5.6292){$3$}
\rput(6,5.6292){$2$}
\rput(6.5,5.6292){$1$}
\rput(7,5.6292){$0$}
\rput(7.5,5.6292){$6$}
\rput(8,5.6292){$5$}
\rput(8.5,5.6292){$4$}

\rput(2.75,5.1962){$3$}
\rput(3.25,5.1962){$1$}
\rput(3.75,5.1962){$6$}
\rput(4.25,5.1962){$4$}
\rput(4.75,5.1962){$2$}
\rput(5.25,5.1962){$0$}
\rput(5.75,5.1962){$5$}
\rput(6.25,5.1962){$3$}
\rput(6.75,5.1962){$1$}
\rput(7.25,5.1962){$6$}
\rput(7.75,5.1962){$4$}
\rput(8.25,5.1962){$2$}

\rput(3,4.7631){$4$}
\rput(3.5,4.7631){$0$}
\rput(4,4.7631){$3$}
\rput(4.5,4.7631){$6$}
\rput(5,4.7631){$2$}
\rput(5.5,4.7631){$5$}
\rput(6,4.7631){$1$}
\rput(6.5,4.7631){$4$}
\rput(7,4.7631){$0$}
\rput(7.5,4.7631){$3$}
\rput(8,4.7631){$6$}

\rput(3.25,4.3301){$4$}
\rput(3.75,4.3301){$3$}
\rput(4.25,4.3301){$2$}
\rput(4.75,4.3301){$1$}
\rput(5.25,4.3301){$0$}
\rput(5.75,4.3301){$6$}
\rput(6.25,4.3301){$5$}
\rput(6.75,4.3301){$4$}
\rput(7.25,4.3301){$3$}
\rput(7.75,4.3301){$2$}

\rput(3.5,3.8971){$0$}
\rput(4,3.8971){$5$}
\rput(4.5,3.8971){$3$}
\rput(5,3.8971){$1$}
\rput(5.5,3.8971){$6$}
\rput(6,3.8971){$4$}
\rput(6.5,3.8971){$2$}
\rput(7,3.8971){$0$}
\rput(7.5,3.8971){$5$}

\rput(3.75,3.4641){$5$}
\rput(4.25,3.4641){$1$}
\rput(4.75,3.4641){$4$}
\rput(5.25,3.4641){$0$}
\rput(5.75,3.4641){$3$}
\rput(6.25,3.4641){$6$}
\rput(6.75,3.4641){$2$}
\rput(7.25,3.4641){$5$}

\rput(4,3.0311){$6$}
\rput(4.5,3.0311){$5$}
\rput(5,3.0311){$4$}
\rput(5.5,3.0311){$3$}
\rput(6,3.0311){$2$}
\rput(6.5,3.0311){$1$}
\rput(7,3.0311){$0$}

\rput(4.25,2.5981){$4$}
\rput(4.75,2.5981){$2$}
\rput(5.25,2.5981){$0$}
\rput(5.75,2.5981){$5$}
\rput(6.25,2.5981){$3$}
\rput(6.75,2.5981){$1$}

\rput(4.5,2.1651){$6$}
\rput(5,2.1651){$2$}
\rput(5.5,2.1651){$5$}
\rput(6,2.1651){$1$}
\rput(6.5,2.1651){$4$}

\rput(4.75,1.7321){$1$}
\rput(5.25,1.7321){$0$}
\rput(5.75,1.7321){$6$}
\rput(6.25,1.7321){$5$}

\rput(5,1.2990){$1$}
\rput(5.5,1.2990){$6$}
\rput(6,1.2990){$4$}

\rput(5.25,0.8660){$0$}
\rput(5.75,0.8660){$3$}

\rput(5.5,0.4330){$3$}

\end{pspicture}
\caption{The Steinhaus triangle $\Delta AP(1,3,20)$}
\end{figure}
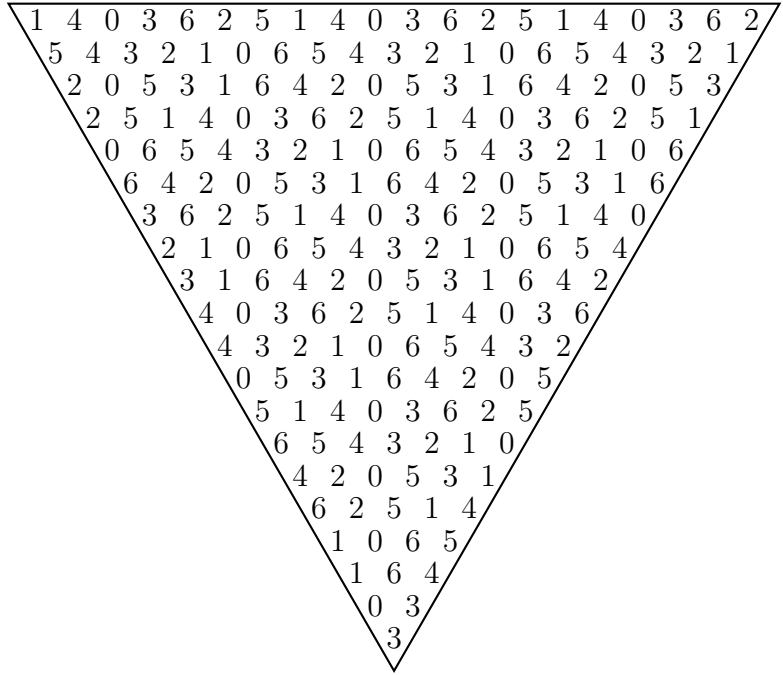

Since there are $n$ distinct elements $a$ in $\Zn$ and $\varphi(n)$ distinct invertible elements $d$ in $\Zn$, it follows that, for each positive integer $m$, there exist exactly $n\varphi(n)$ distinct arithmetic progressions $AP(a,d,m)$ with invertible common difference in $\Zn$ and of length $m$. Therefore, for $n$ odd, Theorem \ref{mainthm1} implies that there exist at least $n\varphi(n)$ balanced sequences of length $m$ for every positive integer $m\equiv0\pmod{\alpha(n)n}$ or $m\equiv-1\pmod{\alpha(n)n}$. However, this is not sufficient to completely settle Molluzzo's Problem, as shown by the following proposition. This shortcoming will be partly overcome in the next section.

\begin{prop}
Let $n>1$ be an odd number. Then
$$
\alpha(n)\geq2.
$$
\end{prop}

\begin{proof}
Let
$$
n = {p_1}^{r_1}\cdots{p_k}^{r_k}
$$
be the prime factorization of the odd number $n>1$. If $\alpha(n)=1$, then
$$
2^{n} \equiv 1 \pmod{n}.
$$
Let $p_j$ be the least prime factor of $n$. Since
$$
2^{n} \equiv 1 \pmod{p_j},
$$
it follows that $\ord_{p_j}(2)$ divides $n$, in contradiction with the fact that $\ord_{p_j}(2)$ divides $p_j-1$ which is relatively prime to $n$.
\end{proof}

\section{The antisymmetric case}

In Section 4, we have seen that there exist infinitely many balanced sequences in $\Zn$ for $n$ odd. More precisely, Theorem \ref{mainthm1} states that all the arithmetic progressions with invertible common difference and of length $m\equiv 0$ or $-1\pmod{\alpha(n)n}$ are balanced. In this section we refine this result by considering the antisymmetric sequences in $\Zn$. This will be sufficient to settle Molluzzo's problem for any $n=3^k$.

Let $X=(x_1,x_2,\ldots,x_m)$ be a finite sequence of length $m\geq1$ in $\Zn$. The sequence $X$ is said to be \textit{antisymmetric} if $x_{m-i+1} = -x_i$, for every integer $1\leq i\leq m$.

We first show that the antisymmetry is preserved by the derivation process and we study the condition to have an antisymmetric primitive sequence of an antisymmetric sequence.

\begin{prop}
Let $X=(x_1,x_2,\ldots,x_m)$ be a finite sequence of length $m\geq1$ in $\Zn$. Then the sequence $X$ is antisymmetric if, and only if, its derived sequence $\partial X$ is also antisymmetric and  $x_{\lceil\frac{m}{2}\rceil} + x_{m-\lceil\frac{m}{2}\rceil+1} = 0$, where $\lceil\frac{m}{2}\rceil$ is the ceiling of $\frac{m}{2}$.
\end{prop}

\begin{proof}
We set $X=(x_1,x_2,\ldots,x_m)$ and $\partial X=Y=(y_1,y_2,\ldots,y_{m-1})$ its derived sequence.
\begin{itemize}
\item[$\Longrightarrow$]
For every integer $1\leq i\leq m-1$, we have
$$
y_{m-i}+y_i = (x_{m-i}+x_{m-i+1})+(x_i+x_{i+1}) = (x_{m-i+1}+x_i)+(x_{m-i}+x_{i+1}) = 0.
$$
\item[$\Longleftarrow$]
By induction, we can prove that
$$
\begin{array}{rl}
x_i & = {(-1)}^{j-i}x_j + \displaystyle\sum_{k=i}^{j-1}{(-1)}^{k-i}y_k,\\
x_j & = {(-1)}^{j-i}x_i + \displaystyle\sum_{k=i}^{j-1}{(-1)}^{j-k-1}y_k,
\end{array}
$$
for all integers $1\leq i<j\leq m$. It follows that
$$
\begin{array}{l}
x_{m-i+1}+x_i = {(-1)}^{\lceil\frac{m}{2}\rceil-i}x_{m-\lceil\frac{m}{2}\rceil+1} + \displaystyle\sum_{k=m-\lceil\frac{m}{2}\rceil+1}^{m-i}{(-1)}^{m-k-i}y_k + {(-1)}^{\lceil\frac{m}{2}\rceil-i}x_{\lceil\frac{m}{2}\rceil}\\
+ \displaystyle\sum_{k=i}^{\lceil\frac{m}{2}\rceil-1}{(-1)}^{k-i}y_k = {(-1)}^{\lceil\frac{m}{2}\rceil-i}\underbrace{\left(x_{\lceil\frac{m}{2}\rceil}+x_{m-\lceil\frac{m}{2}\rceil+1}\right)}_{=0} + \sum_{k=i}^{\lceil\frac{m}{2}\rceil-1}{(-1)}^{k-i}(\underbrace{y_k+y_{m-k}}_{=0})= 0,
\end{array}
$$
for every integer $1\leq i\leq\lceil\frac{m}{2}\rceil-1$.
\end{itemize}
This completes the proof.
\end{proof}

\begin{prop}\label{prop9}
Let $n$ be an odd number. Let $a$ and $d$ be in $\Zn$. Then, the arithmetic progression $AP(a,d,m)$ of length $m\geq2$ is antisymmetric if, and only if, its derived sequence $AP(2a+d,2d,m-1)$ is also antisymmetric.
\end{prop}

\begin{proof}
We set $X=AP(a,d,m)=(x_1,x_2,\ldots,x_m)$ and $\partial X=AP(2a+d,2d,m-1)=(y_1,y_2,\ldots,y_{m-1})$. It follows that
$$
\begin{array}{rl}
y_{m-i} + y_i & = (2a+d)+(m-i-1)2d + (2a+d)+(i-1)2d = 2(2a+(m-1)d)\\
& = 2(a+(m-j)d+a+(j-1)d) = 2(x_{m-j+1}+x_j),
\end{array}
$$
for all integers $1\leq i<m$ and all integers $1\leq j\leq m$.
\end{proof}

In contrast, for $n$ even, this proposition is not true. For instance, for $n=8$, the arithmetic progression $X=(0,1,2,3,4)$ is not antisymmetric in $\Z/8\Z$ but its derived sequence $\partial X=(1,3,5,7)$ is.

We now determine arithmetic progressions which are antisymmetric in $\Zn$ for $n$ odd.

\begin{prop}\label{prop6}
Let $n$ be an odd number. Let $d$ be in $\Zn$ and $m$ be a positive integer. Then, there exists a unique antisymmetric arithmetic progression of length $m$ and with common difference $d$. Moreover, if $m$ is a multiple of $n$, then the unique antisymmetric arithmetic progression with common difference $d$ and of length $m$ is the sequence $AP(2^{-1}d,d,m)$. If $m\equiv-1\pmod{n}$, then the unique antisymmetric arithmetic progression with common difference $d$ and of length $m$ is the sequence $AP(d,d,m)$.
\end{prop}

\begin{proof}
We set $X=AP(a,d,m)=(x_1,x_2,\ldots,x_m)$. If the sequence $X$ is antisymmetric, then
$$
x_{m-i+1}+x_i = 0
$$
for all integers $1\leq i\leq m$. Since
$$
x_{m-i+1}+x_i = a+(m-i)d+a+(i-1)d = 2a+(m-1)d
$$
for each $1\leq i\leq m$, it follows that the arithmetic progression $X$ is antisymmetric if, and only if, $a$, $d$ and the integer $m$ are such that $2a+(m-1)d=0$. Therefore, the sequence
$$
AP\left(2^{-1}(1-m)d,d,m\right)
$$
is the only arithmetic progression of length $m\geq1$ and with common difference $d\in\Zn$ which is antisymmetric. This completes the proof.
\end{proof}

If $n$ is even, the above unicity does not hold in general. For example, in $\Z/8\Z$, the antisymmetric sequences $(0,2,4,6,0)$ and 
$(4,6,0,2,4)$ are both arithmetic progressions of length $m=5$ and of common difference $d=2$.

For every odd number $n$, we denote by $\beta(n)$ \textit{the projective multiplicative order of $2^n$ modulo $n$}, i.e. the smallest positive integer $e$ such that $2^{en}\equiv\pm 1\pmod{n}$, namely
$$
\beta(n) = \min \left\{ e\in\N^*\ \middle|\ 2^{en}\equiv\pm1\pmod{n} \right\}.
$$

Observe that we have the alternative $\alpha(n)=\beta(n)$ or $\alpha(n)=2\beta(n)$. Moreover, $\alpha(n)=2\beta(n)$ if and only if there exists a power $e$ of $2^n$ such that $2^{en}\equiv-1\pmod{n}$. If $n$ is a prime power, then $\beta(n)=\beta(\rad(n))$, in analogy with Proposition \ref{prop8} for $\alpha(n)$.

\begin{prop}\label{prop7}
Let $p$ be an odd prime number. Then,
$$
\beta\left(p^k\right) = \beta\left( p\right),
$$
for every positive integer $k$.
\end{prop}

\begin{proof}
The result follows from the claim that $\alpha(p^k)=2\beta(p^k)$ if and only if $\alpha(p)=2\beta(p)$.

Indeed, if $\alpha(p^k)=2\beta(p^k)$, then we have $2^{\beta(p^k)p^k}\equiv-1\pmod{p^k}$. This implies that $2^{\beta(p^k)p^k}\equiv-1\pmod{p}$ and so $2^{\beta(p^k)p}\equiv-1\pmod{p}$ by Fermat's little theorem. It follows that $\alpha(p)=2\beta(p)$ and $\beta(p)$ divides $\beta(p^k)$.

Conversely, if $\alpha(p)=2\beta(p)$, then $2^{\beta(p)p}\equiv-1\pmod{p}$. By induction on $k$, it follows from the binomial theorem that there exists a positive integer $u_k$ such that $2^{\beta(p)p^k}=-1+u_kp^k$. This leads to the congruence $2^{\beta(p)p^k}\equiv-1\pmod{p^k}$ and so we have $\alpha(p^k)=2\beta(p^k)$ and $\beta(p^k)$ divides $\beta(p)$.

In either of the two cases $\alpha(p^k)=2\beta(p^k)$ or $\alpha(p^k)=\beta(p^k)$, the result follows from Proposition \ref{prop8}.
\end{proof}

We now improve Theorem \ref{mainthm1} by considering the antisymmetric arithmetic progressions with invertible common difference. There are exactly $\varphi(n)$ such sequences, for every length, by Proposition \ref{prop6}.

\begin{thm}\label{mainthm2}
Let $n$ be an odd number and $d$ be an invertible element in $\Zn$. Then
\begin{itemize}
\item
for every $m\equiv0\pmod{\beta(n)n}$, the arithmetic progression $AP(2^{-1}d,d,m)$ is balanced,
\item
for every $m\equiv-1\pmod{\beta(n)n}$, the arithmetic progression $AP(d,d,m)$ is balanced.
\end{itemize}
\end{thm}

The proof is based on Theorem \ref{mainthm1} and on the following lemma.

\begin{lem}\label{lemma4}
Let $n$ be a positive integer and $X$ be an antisymmetric sequence of length $m\geq1$ in $\Zn$. Then we have
$$
\m_{\Delta X}(x) = \m_{\Delta X}(-x),\ \forall x\in\Zn.
$$
\end{lem}

\begin{proof}
By Proposition \ref{prop9}, all the iterated derived sequences of $X$ are antisymmetric. This leads to
$$
\m_{\Delta X}(x) = \sum_{i=0}^{m-1}\m_{\partial^iX}(x) = \sum_{i=0}^{m-1}\m_{\partial^iX}(-x) = \m_{\Delta X}(-x),\ \forall x\in\Zn.
$$
\end{proof}

We are now ready to prove our refinement of Theorem \ref{mainthm1}.

\begin{proof}[Proof of Theorem \ref{mainthm2}]
As in the proof of Theorem \ref{mainthm1}, we derive the case $m\equiv-1\pmod{\beta(n)n}$ from the case $m\equiv0\pmod{\beta(n)n}$. Let $k$ be a positive integer. We set $m=k\beta(n)n-1$ and $X=AP(d,d,m)$. From Proposition \ref{prop1}, the arithmetic progression
$$
Y=AP(2^{-2}d,2^{-1}d,k\beta(n)n)
$$
is a primitive of the sequence $X$. Since $Y$ is an arithmetic progression with invertible common difference $2^{-1}d$ and of length $k\beta(n)n$, it follows from Lemma \ref{lemma3} that each element of $\Zn$ occurs $k\beta(n)$ times in the sequence $Y$. Since $X$ is the derived sequence of $Y$, we have
$$
\m_{\Delta X}(x) = \m_{\Delta\partial Y}(x) = \m_{\Delta Y}(x) - \m_Y(x) = \m_{\Delta Y}(x) - k\beta(n),
$$
for all $x$ in $\Zn$. Therefore, $X$ is balanced if and only if the sequence $Y$ is balanced. This completes the proof of the case $m\equiv -1\pmod{\beta(n)n}$ from the case $m\equiv0\pmod{\beta(n)n}$.

We now settle the case $m\equiv0\pmod{\beta(n)n}$. If $\alpha(n)=\beta(n)$, then this statement is a particular case of Theorem \ref{mainthm1}. Suppose now that $\alpha(n)=2\beta(n)$. Then $2^{\beta(n)n}\equiv-1\pmod{n}$. Let $k$ be a positive integer. We shall show that the sequence
$$
AP\left(2^{-1}d,d,k\beta(n)n\right)
$$
is balanced. We first set
$$
X=AP(2^{-1}d,d,2k\beta(n)n).
$$
We now consider the structure of the Steinhaus triangle $\Delta X$ depicted in Figure~\ref{fig10}. Recall that $\Delta X(i,j)$ denotes the $j$th element of the $i$th row of $\Delta X$, for every integer $1\leq i\leq 2k\beta(n)n$ and every integer $1\leq j\leq 2k\beta(n)n-i+1$.

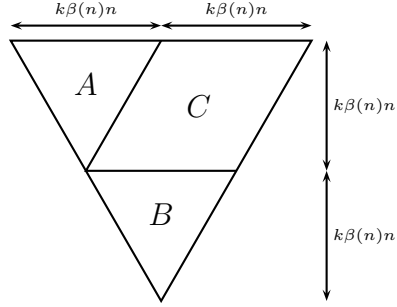
\begin{figure}[!h]
\centering
\begin{pspicture}(4,3.964)
\pspolygon(0,3.464)(4,3.464)(2,0)
\psline(2,3.464)(1,1.732)(3,1.732)
\psline{<->}(0,3.664)(2,3.664)
\rput(1,3.864){\tiny{$k\beta(n)n$}}
\psline{<->}(2,3.664)(4,3.664)
\rput(3,3.864){\tiny{$k\beta(n)n$}}
\psline{<->}(4.2,3.464)(4.2,1.732)
\rput(4.7,2.598){\tiny{$k\beta(n)n$}}
\psline{<->}(4.2,1.732)(4.2,0)
\rput(4.7,0.866){\tiny{$k\beta(n)n$}}
\rput(1,2.8867){$A$}
\rput(2.5,2.598){$C$}
\rput(2,1.1547){$B$}
\end{pspicture}
\caption{\label{fig10}Structure of $\Delta X$}
\end{figure}

The subtriangle $A$ is defined by
$$
A = \left\{ \Delta X(i,j)\ \middle|\ 1\leq i\leq k\beta(n)n\ ,\ 1\leq j\leq k\beta(n)n-i+1\right\}.
$$
Then $A$ is the Steinhaus triangle generated by the $k\beta(n)n$ first elements of $X$, that is,
$$
A = \Delta AP\left(2^{-1}d,d,k\beta(n)n\right).
$$
The subtriangle $B$ is defined by
$$
B = \left\{ \Delta X(i,j)\ \middle|\ k\beta(n)n+1\leq i\leq 2k\beta(n)n\ ,\ 1\leq j\leq 2k\beta(n)n-i+1\right\}.
$$
Then $B$ is the Steinhaus triangle generated by the derived sequence $\partial^{k\beta(n)n}X$, that is,
$$
B = \Delta\partial^{k\beta(n)n}X.
$$
Proposition \ref{prop4} leads to
$$
\partial^{k\beta(n)n}X
\begin{array}[t]{l}
= \partial^{k\beta(n)n}AP\left(2^{-1}d,d,2k\beta(n)n\right)\\
= AP\left( 2^{k\beta(n)n-1}d+2^{k\beta(n)n-1}k\beta(n)nd,2^{k\beta(n)n}d,k\beta(n)n \right).
\end{array}
$$
Since $2^{\beta(n)n}\equiv-1\pmod{n}$, it follows that
$$
\partial^{k\beta(n)n}X = AP\left( {(-1)}^{k}2^{-1}d , {(-1)}^kd , k\beta(n)n \right).
$$
If $k$ is even, then $\partial^{k\beta(n)n}X=AP\left(2^{-1}d,d,k\beta(n)n\right)$ and thus $A=B$. If $k$ is odd, then $\partial^{k\beta(n)n}X=AP\left(-2^{-1}d,-d,k\beta(n)n\right)$. Since it is an antisymmetric arithmetic progression by Proposition \ref{prop6}, it follows from Lemma \ref{lemma4} that $\m_{B}(x)=\m_{B}(-x)$ for all $x$ in $\Zn$. Therefore, we have
$$
\m_B(x)=\m_B(-x)=\m_{\Delta AP(-2^{-1}d,-d,k\beta(n)n)}(-x)=\m_{\Delta AP(2^{-1}d,d,k\beta(n)n)}(x)=\m_A(x)
$$
for all $x\in\Zn$. In all cases, we obtain
$$
\m_B(x) = \m_A(x),\ \forall x\in\Zn.
$$
Finally, the multiset $C$ is defined by
$$
C = \left\{ \Delta X(i,j)\ \middle|\ 1\leq i\leq k\beta(n)n\ ,\ k\beta(n)n-i+2\leq j\leq 2k\beta(n)n-i+1\right\}.
$$
Then each row of $C$ is composed of $k\beta(n)n$ consecutive terms of a derived sequence of $X$. Since, for every $0\leq i\leq 2k\beta(n)n-1$, the derived sequence $\partial^iX$ of $X$ is an arithmetic progression with invertible common difference $2^{i}d$ by Proposition \ref{prop4}, it follows from Lemma \ref{lemma3} that each element of $\Zn$ occurs $k\beta(n)$ times in each row of $C$. Therefore, the multiplicity function of $C$ is the constant function defined by
$$
\m_C(x) = k^2{\beta(n)}^2n,\ \forall x\in\Zn.
$$
Combining the results above, we have
$$
\m_{\Delta X}(x) = \m_A(x)+\m_B(x)+\m_C(x) = 2\m_A(x) + k^2{\beta(n)}^2n
$$
for all $x$ in $\Zn$.

We conclude that the sequence $AP(2^{-1}d,d,k\beta(n)n)$ is balanced if and only if the sequence $X = AP(2^{-1}d,d,2k\beta(n)n) = AP(2^{-1}d,d,k\alpha(n)n)$ is also balanced. This completes the proof of Theorem \ref{mainthm2}.
\end{proof}

We shall now see that this theorem answers in the affirmative Molluzzo's problem in $\Z/3^k\Z$ for all positive integers $k$ and gives a partial answer in the general odd case.

\begin{cor}
Molluzzo's problem is completely solved in $\Z/3^k\Z$ for all positive integers $k$. In other words, there exists a balanced sequence of length $m$ in $\Z/3^k\Z$ if and only if $\binom{m+1}{2}$ is divisible by $3^k$.
\end{cor}

\begin{proof}
Let $k$ be a positive integer. By Proposition \ref{prop7}, we have
$$
\beta(3^k) = \beta(3) = 1.
$$
Let $d$ be an invertible element in $\Z/3^k\Z$. Then, Theorem \ref{mainthm2} implies that 
\begin{itemize}
\item
$AP(2^{-1}d,d,m)$ is balanced for every positive integer $m\equiv0\pmod{3^k}$,
\item
$AP(d,d,m)$ is balanced for every positive integer $m\equiv-1\pmod{3^k}$.
\end{itemize}
Finally, from Corollary \ref{cor1}, we know that $3^k$ divides the binomial coefficient $\binom{m+1}{2}$ if, and only if, the positive integer $m$ is congruent to $0$ or $-1$ modulo $3^k$. Therefore, we have constructed balanced sequences for all admissible lengths in $\Z/3^k\Z$.
\end{proof}

For every odd number $n$, the results above, namely Theorem \ref{thm6} and Theorem \ref{mainthm2}, partly solve Molluzzo's problem in $\Zn$ in the exact proportion of $\frac{1}{2^{\omega(n)-1}\beta(n)}$, where $\omega(n)$ is the number of distinct prime factors of $n$. Indeed, if we consider the sets
$$
N(n) = \left\{m\in\N\ \middle|\ \binom{m+1}{2}\equiv0\pmod{n}\right\},
$$
and
$$
B(n) = \left\{m\in\N\ \middle|\ \exists\ \text{a\ balanced\ sequence\ in}\ \Zn\ \text{of\ length}\ m\right\},
$$
then clearly $B(n)\subset N(n)$ as pointed out in Sections 1 and 2. Moreover, Molluzzo's problem can be reformulated as the question whether $B(n)=N(n)$ for all $n>1$.

It follows from Theorem \ref{thm6} and Theorem \ref{mainthm2} that
$$
\frac{|B(n)\cap\llbracket0,k\rrbracket|}{|N(n)\cap\llbracket0,k\rrbracket|}\geq\frac{1}{2^{\omega(n)-1}\beta(n)},
$$
for all $k\geq\beta(n)n$. Since $2^{\omega(n)-1}\beta(n)\geq2$ for every odd number $n\neq3^k$, it follows that our method gives a complete solution to Molluzzo's Problem for the powers of three only. For example, for $n=5^k$, we have $2^{\omega(n)-1}\beta(n)=2$, whence our results in this case produce balanced sequences for half of the admissible lengths.

\section{Balanced arithmetic progressions in $\Zn$ for $n$ even}

In preceding sections we have seen that, for any odd number $n$ and any invertible element $d$ in $\Zn$, the arithmetic progressions $AP(a,d,m)$, for $m\equiv 0$ or $-1\pmod{\alpha(n)n}$, constitute an infinite family of balanced sequences. Here we study the case where $n$ is even and show that, in contrast, arithmetic progressions are almost never balanced.

\begin{thm}
Let $n$ be an even number and $a$ and $d$ be in $\Zn$. Then the arithmetic progression $X=AP(a,d,m)$ is balanced if, and only if, we have
$$
\left\{
\begin{array}{ccl}
n=2 & \text{and} & X\in\left\{ (0,1,0) , (1,1,1) , (0,1,0,1) , (1,0,1,0) \right\},\\
\text{or} &  & \\
n=6 & \text{and} & X\in\left\{ (1,3,5) , (2,3,4) , (4,3,2) , (5,3,1) \right\}.
\end{array}
\right.
$$
\end{thm}

\begin{proof}
Suppose that the arithmetic progression $X=AP(a,d,m)$ is balanced. We first consider the canonical surjective morphism $\pi_2 : \Zn\twoheadlongrightarrow\Z/2\Z$ and the projected sequence $\pi_2(X)=AP\left(\pi_2(a),\pi_2(d),m\right)$ in $\Z/2\Z$ which is also balanced by Theorem \ref{thm2}. If we denote by $\Delta\pi_2\left(X\right)\left(i,j\right)$ the $j$th element of the $i$th row of $\Delta\pi_2\left(X\right)$, then Proposition \ref{prop5} implies that
$$
\Delta\pi_2(X)\left(i,j\right) = 2^{i-2}\left( 2\pi_2(a)+(2j+i-3)\pi_2(d) \right) = 0 \in \Z/2\Z,
$$
for all $i\geq3$. Therefore, for every $i\geq3$, the derived sequence $\partial^iX$ only contains zeros. Since the sequence $\pi_2\left(X\right)$ is balanced, it follows that its triangle $\Delta\pi_2\left(X\right)$ contains at least twice as many elements as $\pi_2\left(X\right)$ and its derived sequence $\partial\pi_2\left(X\right)$ and hence, the positive integer $m$ is solution of the inequality
$$
2\binom{m-1}{2}\leq\binom{m+1}{2}.
$$
Therefore $m\in\llbracket1,6\rrbracket$. Moreover, the necessary condition that the binomial coefficient $\binom{m+1}{2}$, the cardinality of the Steinhaus triangle $\Delta\pi_2\left(X\right)$, is even implies that $m=3$ or $m=4$. We now distinguish the different cases.
\begin{itemize}
\item[$\mathbf{m=3:}$]
Since $n$ divides the binomial coefficient $\binom{m+1}{2}=6$, it follows that $n=2$ or $n=6$.
\begin{itemize}
\item[$\mathbf{n=2:}$]
There exist four arithmetic progressions of length $m=3$ in $\Z/2\Z$ including two that are balanced, the sequences $X_1=(0,1,0)$ and $X_2=(1,1,1)$.
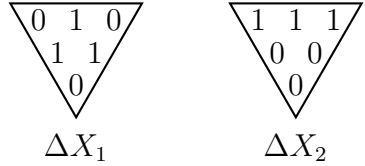
\begin{center}
\begin{tabular}{cc}
\begin{pspicture}(2.5,1.732)
\pspolygon(0.375,1.5155)(2.125,1.5155)(1.25,0)
\rput(0.75,1.299){$0$}
\rput(1.25,1.299){$1$}
\rput(1.75,1.299){$0$}
\rput(1,0.866){$1$}
\rput(1.5,0.866){$1$}
\rput(1.25,0.433){$0$}
\end{pspicture}
&
\begin{pspicture}(2.5,1.732)
\pspolygon(0.375,1.5155)(2.125,1.5155)(1.25,0)
\rput(0.75,1.299){$1$}
\rput(1.25,1.299){$1$}
\rput(1.75,1.299){$1$}
\rput(1,0.866){$0$}
\rput(1.5,0.866){$0$}
\rput(1.25,0.433){$0$}
\end{pspicture}\\
$\Delta X_1$ & $\Delta X_2$
\end{tabular}
\end{center}
\item[$\mathbf{n=6:}$]
We look for a Steinhaus triangle $\Delta X$ containing each element of $\Z/6\Z$ once. Since the equality $\Delta X\left(i,j\right)=0$ implies $\Delta X(i,j-1)=\Delta X(i+1,j-1)$ or $\Delta X(i,j+1)=\Delta X(i+1,j)$, it follows that 
$\Delta X(3,1)=0$ and hence, we have
$$
0=\Delta X(3,1) = 4(a+d) = 4\Delta X(1,2).
$$
Therefore, $\Delta X(1,2)=3$ and we look for balanced arithmetic progressions
$$
X=(a,3,-a),
$$
with $a\in\{1,2,4,5\}$. Finally, the four arithmetic progressions $X_3=(1,3,5)$, $X_4=(2,3,4)$, $X_5=(4,3,2)$ and $X_6=(5,3,1)$ are balanced.
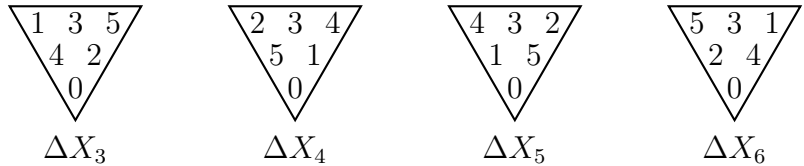
\begin{center}
\begin{tabular}{cccc}
\begin{pspicture}(2.5,1.732)
\pspolygon(0.375,1.5155)(2.125,1.5155)(1.25,0)
\rput(0.75,1.299){$1$}
\rput(1.25,1.299){$3$}
\rput(1.75,1.299){$5$}
\rput(1,0.866){$4$}
\rput(1.5,0.866){$2$}
\rput(1.25,0.433){$0$}
\end{pspicture}
&
\begin{pspicture}(2.5,1.732)
\pspolygon(0.375,1.5155)(2.125,1.5155)(1.25,0)
\rput(0.75,1.299){$2$}
\rput(1.25,1.299){$3$}
\rput(1.75,1.299){$4$}
\rput(1,0.866){$5$}
\rput(1.5,0.866){$1$}
\rput(1.25,0.433){$0$}
\end{pspicture}
&
\begin{pspicture}(2.5,1.732)
\pspolygon(0.375,1.5155)(2.125,1.5155)(1.25,0)
\rput(0.75,1.299){$4$}
\rput(1.25,1.299){$3$}
\rput(1.75,1.299){$2$}
\rput(1,0.866){$1$}
\rput(1.5,0.866){$5$}
\rput(1.25,0.433){$0$}
\end{pspicture}
&
\begin{pspicture}(2.5,1.732)
\pspolygon(0.375,1.5155)(2.125,1.5155)(1.25,0)
\rput(0.75,1.299){$5$}
\rput(1.25,1.299){$3$}
\rput(1.75,1.299){$1$}
\rput(1,0.866){$2$}
\rput(1.5,0.866){$4$}
\rput(1.25,0.433){$0$}
\end{pspicture}\\
$\Delta X_3$ & $\Delta X_4$ & $\Delta X_5$ & $\Delta X_6$
\end{tabular}
\end{center}
\end{itemize}
\item[$\mathbf{m=4:}$]
Since $n$ divides the binomial coefficient $\binom{m+1}{2}=10$, it follows that $n=2$ or $n=10$.
\begin{itemize}
\item[$\mathbf{n=2:}$]
There exist four arithmetic progressions of length $m=4$ in $\Z/2\Z$ including two that are balanced, the sequences $X_7=(0,1,0,1)$ and $X_8=(1,0,1,0)$.
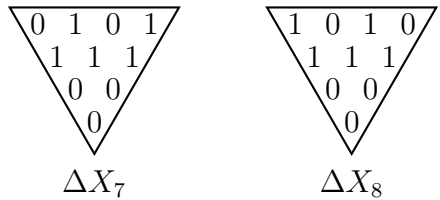
\begin{center}
\begin{tabular}{cc}
\begin{pspicture}(3,2.165)
\pspolygon(0.375,1.9485)(2.625,1.9485)(1.5,0)
\rput(0.75,1.732){$0$}
\rput(1.25,1.732){$1$}
\rput(1.75,1.732){$0$}
\rput(2.25,1.732){$1$}
\rput(1,1.299){$1$}
\rput(1.5,1.299){$1$}
\rput(2,1.299){$1$}
\rput(1.25,0.866){$0$}
\rput(1.75,0.866){$0$}
\rput(1.5,0.433){$0$}
\end{pspicture}
&
\begin{pspicture}(3,2.165)
\pspolygon(0.375,1.9485)(2.625,1.9485)(1.5,0)
\rput(0.75,1.732){$1$}
\rput(1.25,1.732){$0$}
\rput(1.75,1.732){$1$}
\rput(2.25,1.732){$0$}
\rput(1,1.299){$1$}
\rput(1.5,1.299){$1$}
\rput(2,1.299){$1$}
\rput(1.25,0.866){$0$}
\rput(1.75,0.866){$0$}
\rput(1.5,0.433){$0$}
\end{pspicture}\\
$\Delta X_7$ & $\Delta X_8$
\end{tabular}
\end{center}
\item[$\mathbf{n=10:}$]
We look for a Steinhaus triangle $\Delta X$ containing each element of $\Z/10\Z$ once. Since the equality $\Delta X\left(i,j\right)=0$ implies $\Delta X(i,j-1)=\Delta X(i+1,j-1)$ or $\Delta X(i,j+1)=\Delta X(i+1,j)$, it follows that 
$\Delta X(4,1)=0$ and hence, we have
$$
0 = \Delta X(4,1) = 4(2a+3d) = 4\Delta X(2,2).
$$
Therefore, $\Delta X(2,2)=5$. Moreover, if $2a+3d=5$, then
$$
d = 3^{-1}(5-2a) = 7(5-2a) = 5-4a.
$$
Thus, we look for balanced arithmetic progressions
$$
X=(a,5-3a,3a,5-a)
$$
with $a\in\{1,2,3,4,6,7,8,9\}$. Finally, there is no balanced arithmetic progression in $\Z/10\Z$.

\end{itemize}
\end{itemize}
\end{proof}

\section{Concluding remarks and open subproblems}

We have seen, throughout Sections 4 and 5, that in $\Zn$, for $n$ odd, arithmetic progressions with invertible common difference give infinitely many balanced sequences. Particularly, in every $\Z/3^k\Z$, they yield a full solution to Molluzzo's problem. In Section 6, we have proved that arithmetic progressions are almost never balanced in $\Zn$ for $n$ even. The following \textit{particular cases of Molluzzo's problem} remain open and are of particular interest.

\begin{prob2}
Do there exist infinitely many balanced sequences in $\Zn$ for every even $n\geq4$?
\end{prob2}

\begin{prob2}
Let $n$ be an odd number. Does there exist a balanced sequence of length $m$ for every multiple $m$ of $n$?
\end{prob2}

There are some indications that Problem 2 may be more tractable than the full Molluzzo's problem.

\section*{Acknowledgments}

The author would like to thank Prof. Shalom Eliahou for introducing him to the subject and for his help in preparing this paper.

\nocite{*}
\bibliographystyle{plain}
\bibliography{mabiblio}

\noindent Jonathan Chappelon\\
Laboratoire de Math\'ematiques Pures et Appliqu\'ees Joseph Liouville, FR 2956 CNRS\\
Universit\'e du Littoral C\^ote d'Opale\\
50 rue F. Buisson, B.P. 699, F-62228 Calais Cedex, France\\
e-mail: jonathan.chappelon@lmpa.univ-littoral.fr

\end{document}